\documentclass[11pt,reqno]{amsart}
\usepackage{amsmath,amssymb}
\usepackage{dsfont}
\usepackage{xcolor}
\usepackage{graphicx}
\usepackage{caption}
\usepackage{subcaption}
\usepackage{mathtools}
\usepackage{bm}
\usepackage[margin=1in]{geometry}
\usepackage[colorlinks=true,citecolor=magenta,linkcolor=blue]{hyperref}

\newtheorem{theorem}{Theorem}
\newtheorem{lemma}{Lemma}
\newtheorem{proposition}{Proposition}
\newtheoremstyle{boldremark}  
  {}                         
  {}                         
  {}                         
  {}                         
  {\bfseries}               
  {.}                       
  { }                       
  {}                        

\theoremstyle{boldremark} 
\newtheorem{remark}{Remark}
\newtheorem{definition}{Definition}

\makeatletter
\def\blfootnote{\xdef\@thefnmark{}\@footnotetext}
\makeatother

\author[J. Borsotti]{Jacopo Borsotti}
\address{Jacopo Borsotti \hfill\break
	Department of Mathematical, Physical and Computer Sciences \hfill\break
	University of Parma \hfill\break
	Parco Area delle Scienze 53/A, 43124 Parma, Italy \vspace*{2mm}}
\email{jacopo.borsotti@unipr.it \vspace*{5mm}}

\title[Hospitalizations and economic impact SIRS]{An SIRS model with hospitalizations: \\ economic impact by disease severity}

\begin{document}

\maketitle

\vspace*{-0.8cm}
\begin{abstract}
We introduce a two-timescale SIRS-type model in which a fraction $\theta$ of infected individuals experiences a severe course of the disease, requiring hospitalization. During hospitalization, these individuals do not contribute to further infections. We analyze the model's equilibria, perform a bifurcation analysis, and explore its two-timescale nature (using techniques from Geometric Singular Perturbation Theory). Our main result provides an explicit expression for the value of $\theta$ that maximizes the total number of hospitalized individuals for long times, revealing that this fraction can be lower than 1. This highlights the interesting effect that a severe disease, by necessitating widespread hospitalization, can indirectly suppress contagions and, consequently, reduce hospitalizations. Numerical simulations illustrate the growth in the number of hospitalizations for short times. The model can also be interpreted as a scenario where only a fraction $\theta$ of infected individuals develops symptoms and self-quarantines.
\end{abstract}

\tableofcontents

\textbf{Keywords:} Fast–slow system, Geometric singular perturbation theory, Entry–exit function, Epidemic model, Hospitalizations forecasting, Asymptomatic transmission

\section{Introduction}

\noindent The origins of mathematical epidemic modeling based on compartmental models trace back to the pioneering work of Kermack and McKendrick \cite{5}. Since that time, many models have been introduced in order to account for several factors and simulate complex epidemic dynamics \cite{MR4362890}.

A key distinction in epidemiological models is between the SIR (Susceptible-Infected-Recovered) and SIRS (Susceptible-Infected-Recovered-Susceptible) models. The former describes infections that grant permanent immunity, while the latter represents situations where recovered individuals gradually lose their immunity. Several epidemiological models exhibit a multi-timescale structure. For instance, in the SIRS model, the rate of immunity loss is typically much slower than the rates of infection and recovery. Another example includes models that incorporate both disease and demographic dynamics: the infectious period is generally much shorter than the average lifespan of individuals inside the population \cite{9, 3} and the individual behavior might change much faster than the spread of the epidemic \cite{13, 10, 11, 14,LI2025116273, 12}. These differences in timescales can be analyzed using the Quasi-Steady-State Approximation \cite{8, 7} or Geometric Singular Perturbation Theory \cite{15,16,17,12}.

The emergence and re-emergence of infectious diseases pose a serious threat to public health and can lead to significant economic and social consequences \cite{6}. The direct costs of epidemics are closely related to the number of hospitalizations or the size of the quarantined population. 
Over the years, optimal control models have been widely applied to create effective strategies to mitigate the impact of epidemics, for example by reducing the number of contagions or medical costs. These models generally assume the possibility of controlling part of the population, for example, by vaccinating susceptible individuals or isolating infected ones \cite{27,22,19,18,21,20, 25, 26}. 

In the present work, we introduce a novel SIRS-type model in which a fraction $\theta$ of infected individuals experiences a severe course of the disease, requiring hospitalization. The main goal of this paper is to understand the effects of the parameter $\theta$. In particular, we will focus on the relationship between the severity of the disease, measured in terms of $\theta$, and the economic impact, measured in terms of hospitalizations. It is not obvious that the total number of hospitalizations increases as $\theta$ increases. Indeed, when hospitalized, an infected individual is no longer able to spread the disease. Hence, mass hospitalizations related to a severe disease could indirectly reduce the spread of the epidemic. In fact, the main result of this paper is that the value of $\theta$ that maximizes hospitalizations can be lower than $1$ (obviously, $\theta=0$ minimizes them). Since we do not implement a control strategy, our model is not an optimal control model. Our objective is to try to maximize a quantity by varying a property of the disease (which, once fixed, remains constant for the entire duration of the epidemic). Therefore, our model is somewhat related to an optimal control model; however, we focus only on the characteristics of the disease rather than implementing a control strategy. This represents a great strength of our model. Indeed, control strategies are sometimes extremely difficult to apply in a real society. We hope that the results presented in this paper may help to better understand the influence that the characteristics of a disease have on the spread of an epidemic. 

Since we assume that the rate of immunity loss is much smaller than the rates of infection and recovery, our model has a two-timescale nature. These dynamics are mainly studied using Geometric Singular Perturbation Theory \cite{23,24,40,0}, which helps us understand the behavior of the orbits between two different waves of the epidemic. Interestingly, the two-timescale nature of the model simplifies the study of the economic impact related to the spread of the epidemic. Indeed, we will be able to simplify the expression of the endemic equilibrium and, consequently, subsequent calculations. 

Finally, we highlight that the model we are about to introduce and study can also be interpreted as a scenario in which only a fraction $\theta$ of infected individuals develops symptoms and self-quarantines \cite{10418977,32,30,31,29,33}. In this case, the economic impact can be measured in terms of the number of quarantined individuals.

This paper is organized as follows. In Section \ref{section1} we introduce the novel model and describe a realistic choice of its initial conditions. In Section \ref{section2} we analyze the disease-free and endemic equilibria of the model, and we perform a bifurcation analysis. In Section \ref{section3} we study the two-timescale nature of our model. Section \ref{section4} presents the main results of this paper regarding the relationship between disease severity and the economic impact. We conclude this work in Section \ref{section5}. Numerical results are presented in both Sections \ref{section3} and \ref{section4}. 
	
\section{The model} \label{section1}

\noindent In this section, we propose a novel compartmental SIRS model with hospitalizations. We partition the total population in five compartments, with respect to an ongoing epidemic:
\begin{itemize}
    \item $S$ represents the susceptible individuals; 
    \item $I$ represents individuals characterized by a standard course of the disease; 
    \item $C$ represents individuals characterized by a critical course of the disease who will therefore require hospital treatment shortly; 
    \item $H$ represents hospitalized individuals who therefore cannot infect susceptible individuals; 
    \item $R$ represents individuals who have recovered from the infection and are completely immune. 
\end{itemize}
Denote with $N=S+I+C+H+R$ the total population. The model is described by the following system of ODEs (see Figure \ref{f1}): 
\begin{align} \label{eq1}
\left\{
\begin{aligned}
    \dot{S}(t) &= -\beta \frac{S(t)}{N(t)}  (I(t)+C(t)) + \varepsilon R(t), \\ 
    \dot{I}(t) &= (1-\theta) \beta \frac{S(t)}{N(t)}  (I(t)+C(t)) - \gamma_I I(t), \\ 
    \dot{C}(t) &= \theta \beta \frac{S(t)}{N(t)}(I(t)+C(t)) - \gamma_C C(t), \\ 
    \dot{H}(t) &= \gamma_C C(t) - \gamma_H H(t), \\ 
    \dot{R}(t) &= \gamma_I I(t) + \gamma_H H(t) - \varepsilon R(t), 
\end{aligned}
\right.
\end{align}
where the $^\cdot$ indicates the derivative with respect to the time $t$. The parameters of the system are the following: 
\begin{itemize}
    \item $\beta>0$ is the rate at which the susceptible are infected by individuals characterized by a standard or critical course of the disease; 
    \item $\theta \in [0,1]$ represents the probability that an infected individuals will undergo a critical course of the disease; 
    \item $\gamma_I>0$ is the recovery rate of individuals characterized by a standard course of the disease; 
    \item $\gamma_C > 0$ is the rate at which infected people subject to a critical trend of the disease need to be subjected to hospital care; 
    \item $\gamma_H > 0$ is the recovery rate of individuals subjected to hospital care; 
    \item $0 < \varepsilon \ll 1$  is the loss rate of complete immunity, in particular $\varepsilon \ll \beta, \gamma_I, \gamma_C, \gamma_H$. 
\end{itemize}
Observe that the total population $N$ is conserved over time, hence we can divide all compartments by $N$, which is equivalent to supposing $N=1$. Moreover, we will assume that 
\begin{equation*}
    \gamma_I \le \gamma_C
\end{equation*}
since we expect that severely infected individuals will require hospital treatment before standard infected individuals are recovered. For the sake of simplicity, we did not include a mortality rate among hospitalized individuals. However, when analyzing hospitalization costs, we will focus also on the first wave of the epidemic. During this period, mortality does not affect the costs estimate, as both deceased and recovered individuals do not contribute to disease transmission. Moreover, with this choice our model can also be interpreted as a scenario in which only a fraction $\theta$ of infected individuals develops symptoms and self-quarantines. System (\ref{eq1}) evolves in the biologically relevant region 
\begin{equation*}
    \Bar{\Delta} \coloneqq \{(S,I,C,H,R) \in \mathbb{R}^5 \colon \; S,I,C,H,R \ge 0, \; S+I+C+H+R = 1\}  
\end{equation*}
and on two different timescales: the fast timescale $t$, and the slow timescale $\tau=\varepsilon t$ (this two-timescale structure will be made explicit in Section \ref{section3}). 

\begin{figure}
    \includegraphics[width=12cm]{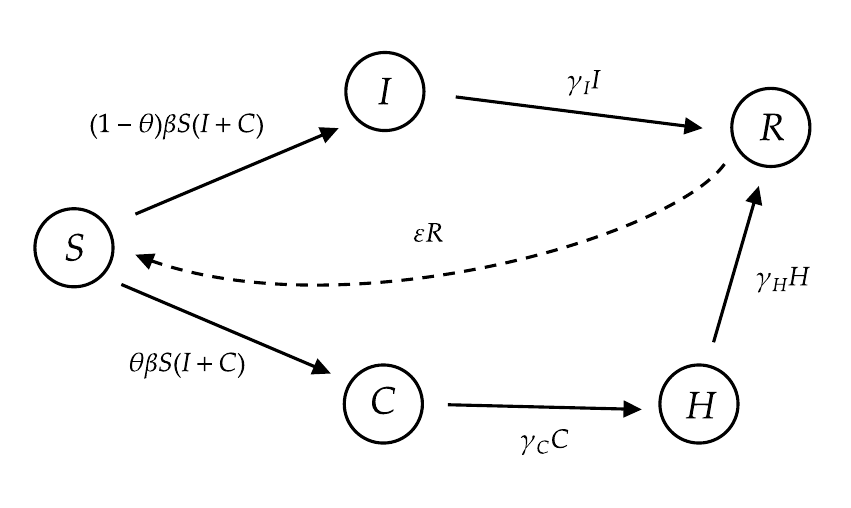}
    \caption{Flow of system (\ref{eq1}). Solid lines represent fast processes (infections, recoveries, and hospitalizations), while the dotted line represents the slow process (loss of complete immunity). This distinction between fast and slow processes highlights the presence of two different timescales, which influence the system’s evolution by separating rapid epidemic events from the gradual loss of immunity.} \label{f1}
\end{figure} 

In the following, we drop the dependence of the compartments $S$, $I$, $C$, $H$, and $R$ on the time variables for ease of notation. We will specify whenever the time variable is changed as a consequence of time rescaling. Since the total population is constant, we can reduce the dimensionality of the system from 5 to 4 via $R=1-S-I-C-H$ obtaining 
\begin{align} \label{eq2}
\left\{
\begin{aligned}
    \dot{S} &= -\beta S (I+C) + \varepsilon (1-S-I-C-H), \\ 
    \dot{I} &= (1-\theta)\beta S (I+C) - \gamma_I I, \\ 
    \dot{C} &= \theta \beta S(I+C) - \gamma_C C, \\ 
    \dot{H} &= \gamma_C C - \gamma_H H.
\end{aligned}
\right.
\end{align}
System (\ref{eq2}) evolves in the biologically relevant region 
\begin{equation} \label{eq3}
    \Delta \coloneqq \{(S,I,C,H) \in \mathbb{R}^4 \colon \; S,I,C,H\ge 0, \; S+I+C+H \le 1\}, 
\end{equation}
as proven in the following proposition.

\begin{proposition}
    The set $\Delta$ (\ref{eq3}) is forward invariant for orbits of system (\ref{eq2}). 
\end{proposition}
\begin{proof}
    For $X \in \{S,I,C,H\}$ we have 
    \begin{equation*}
        \dot{X}|_{X=0} \ge 0.
    \end{equation*}
    Moreover, if we set $Y=S+I+C+H$ we get 
    \begin{equation*}
        \dot{Y}|_{Y=1} = - \gamma_I I - \gamma_H H \le 0, 
    \end{equation*}
    which was to be expected, since the only outwards flows from $Y$ are $-\gamma_I I $ and $-\gamma_H H$ to the $R$ compartment. 
\end{proof}

\subsection{Initial conditions} \label{subsec1}

\noindent We conclude the section with some remarks concerning the choice of realistic initial conditions $\mathbf{x_0}=(S_0, I_0, C_0, H_0)$ for system (\ref{eq2}). When the spread of the disease begins, we should assume that $H_0=R_0=0$, which implies that $S_0=1-I_0-C_0$. Moreover, we must have $I_0+C_0 \ll S_0$. In particular, from the definition of $\theta$ we directly have $I_0=(1-\theta)(I_0+C_0)$ and $C_0=\theta(I_0+C_0)$. Therefore it suffices to choose the (very small) initial number of infective individuals to deduce the initial conditions.     

\begin{remark}
If $\gamma_I=\gamma_C=\gamma$ and the initial conditions are chosen in the way just described, then it is easy to show that $I \equiv (1-\theta) T$ and $C \equiv \theta T$, where we defined $T=I+C$. Therefore system (\ref{eq2}) reduces to 
\begin{align} \label{eqsimple}
\left\{
\begin{aligned}
    \dot{S} &= -\beta S T + \varepsilon (1-S-T-H), \\ 
    \dot{T} &= \beta S T - \gamma T, \\ 
    \dot{H} &= \gamma \theta T - \gamma_H H ,
\end{aligned}
\right.
\end{align}
which has dimension 3 instead of dimension 4. 
\end{remark}

\section{Equilibria, stability, and bifurcation analysis} \label{section2}

\subsection{Disease-free equilibrium}

\noindent The disease-free equilibrium (DFE) of (\ref{eq2}) is $\mathbf{x_1}=(1,0,0,0)$. The Next Generation Matrix method \cite{1} allows us to find the value of the basic reproduction number $\mathcal{R}_0$, which biologically represents the expected number of cases directly generated by one infective individual in a large population where all other people are susceptible to infection.

\begin{proposition} \label{prop1}
    The basic reproduction number of system (\ref{eq2}) is 
    \begin{equation*}
        \mathcal{R}_0=\beta\left(\frac{1-\theta}{\gamma_I}+\frac{\theta}{\gamma_C}\right). 
    \end{equation*}
\end{proposition}
\begin{proof}
    System (\ref{eq2}) has three disease compartments, namely $I$, $C$, and $H$. We can write 
    \begin{align*}
        \dot{I} = \mathcal{F}_1(\mathbf{x})-\mathcal{V}_1(\mathbf{x}), \\ 
        \dot{C} = \mathcal{F}_2(\mathbf{x})-\mathcal{V}_2(\mathbf{x}), \\ 
        \dot{H} = \mathcal{F}_3(\mathbf{x})-\mathcal{V}_3(\mathbf{x}),  
    \end{align*}
    where $\mathbf{x}=(S,I,C,H)$ and 
    \begin{align*}
        \mathcal{F}_1(\mathbf{x}) = (1-\theta)\beta S (I+C),& \quad \mathcal{V}_1(\mathbf{x})=\gamma_I I, \\ 
        \mathcal{F}_2(\mathbf{x}) = \theta\beta S (I+C),& \quad \mathcal{V}_2(\mathbf{x})=\gamma_C C, \\
        \mathcal{F}_3(\mathbf{x}) = 0,& \quad \mathcal{V}_3(\mathbf{x})=\gamma_H H - \gamma_C C. 
    \end{align*}
    Thus we obtain 
    \begin{equation*}
    F = 
    \begin{bmatrix}
    \frac{\partial \mathcal{F}_1}{\partial I} (\mathbf{x_1}) & \frac{\partial \mathcal{F}_1}{\partial C} (\mathbf{x_1})  & \frac{\partial \mathcal{F}_1}{\partial H} (\mathbf{x_1})\\
    \frac{\partial \mathcal{F}_2}{\partial I} (\mathbf{x_1}) & \frac{\partial \mathcal{F}_2}{\partial C} (\mathbf{x_1})  & \frac{\partial \mathcal{F}_2}{\partial H} (\mathbf{x_1})\\
    \frac{\partial \mathcal{F}_3}{\partial I} (\mathbf{x_1}) & \frac{\partial \mathcal{F}_3}{\partial C} (\mathbf{x_1})  & \frac{\partial \mathcal{F}_3}{\partial H} (\mathbf{x_1})\\
    \end{bmatrix}
    = 
    \begin{bmatrix}
        (1-\theta)\beta & (1-\theta) \beta & 0 \\ 
        \theta \beta & \theta \beta & 0 \\ 
        0 & 0 & 0 \\
    \end{bmatrix}
    \end{equation*}
    and 
    \begin{equation*}
    V = 
    \begin{bmatrix}
    \frac{\partial \mathcal{V}_1}{\partial I} (\mathbf{x_1}) & \frac{\partial \mathcal{V}_1}{\partial C} (\mathbf{x_1})  & \frac{\partial \mathcal{V}_1}{\partial H} (\mathbf{x_1})\\
    \frac{\partial \mathcal{V}_2}{\partial I} (\mathbf{x_1}) & \frac{\partial \mathcal{V}_2}{\partial C} (\mathbf{x_1})  & \frac{\partial \mathcal{V}_2}{\partial H} (\mathbf{x_1})\\
    \frac{\partial \mathcal{V}_3}{\partial I} (\mathbf{x_1}) & \frac{\partial \mathcal{V}_3}{\partial C} (\mathbf{x_1})  & \frac{\partial \mathcal{V}_3}{\partial H} (\mathbf{x_1})\\
    \end{bmatrix}
    = 
    \begin{bmatrix}
        \gamma_I & 0 & 0 \\ 
        0 & \gamma_C & 0 \\ 
        0 & -\gamma_C & \gamma_H \\
    \end{bmatrix}. 
    \end{equation*}
    Therefore the Next Generation Matrix $M$, defined as $M=F V^{-1}$, is \cite{1}
    \begin{equation*}
        M=
        \begin{bmatrix}
            \beta \frac{1-\theta}{\gamma_I} & \beta \frac{1-\theta}{\gamma_C} & 0 \\ 
            \beta \frac{\theta}{\gamma_I} & \beta \frac{\theta}{\gamma_C} & 0 \\ 
            0 & 0 & 0 \\ 
        \end{bmatrix},
    \end{equation*}
    from which 
    \begin{equation*}
        \mathcal{R}_0=\rho(M)=\beta\left(\frac{1-\theta}{\gamma_I}+\frac{\theta}{\gamma_C}\right),
    \end{equation*}
    where $\rho(\cdot)$ denotes the spectral radius of a matrix. 
\end{proof}

\begin{remark}
    Notice that the basic reproduction number does not depend on the dynamics of hospitalized individuals $H$. Indeed, even if they contracted the disease, they are not able to infect susceptible individuals. Moreover, if $\gamma_I=\gamma_C=\gamma$ then $\mathcal{R}_0=\beta/\gamma$, which is independent of $\theta$. 
\end{remark}

A direct consequence of Proposition \ref{prop1} is an (in)stability result for the DFE. To prove such result, we first need the following lemma. 

\begin{lemma} \label{lemma1}
        If $\mathcal{R}_0<1$, then $\beta < \gamma_I + \gamma_C$. 
    \end{lemma}
    \begin{proof}
        Suppose by contradiction that $\beta \ge \gamma_I+\gamma_C$, then 
        \begin{equation*}
        \begin{split}
            \mathcal{R}_0 & \ge \frac{\gamma_I + \gamma_C}{\gamma_I \gamma_C} (\gamma_C(1-\theta)+\gamma_I \theta) \ge \frac{\gamma_I + \gamma_C}{\gamma_I \gamma_C} \min{\{\gamma_I, \gamma_C\}} \\ 
            &=\frac{\gamma_I + \gamma_C}{\max{\{\gamma_I, \gamma_C\}}} = 1+ \frac{\min{\{\gamma_I, \gamma_C\}}} {\max{\{\gamma_I, \gamma_C\}}} \\ 
            & > 1, 
        \end{split}
        \end{equation*}
        which contradicts our hypothesis. 
\end{proof}

\begin{theorem} \label{teo1}
    The DFE $\mathbf{x_1}$ is globally asymptotically stable if $\mathcal{R}_0 < 1$, while it is unstable if $\mathcal{R}_0 > 1$. 
\end{theorem}
\begin{proof}
    Assume that $\mathcal{R}_0 < 1$. We would like to construct a Lyapunov function $u$ for system (\ref{eq2}) and prove the global asymptotic stability of the DFE using LaSalle's Invariance Principle. Following \cite{4}, we define 
    \begin{equation*} 
        u = \frac{I}{\gamma_I} + \frac{C}{\gamma_C}, 
    \end{equation*}
    which satisfies 
    \begin{equation} \label{derLyap}
        \dot{u} = (I+C) (S \mathcal{R}_0-1) \le 0. 
    \end{equation}
    Note that $\dot{u}|_{\mathbf{x}_1}=0$, but we also have $\dot{u}=0$ if at some time $I+C=0$. LaSalle's Invariance Principle implies that all orbits converge towards the largest invariant subset of the set of the equilibria of (\ref{derLyap}), which is $\{I=C=0\}$ since if $I+C=0$ then $I \equiv C \equiv 0$. On such set, (\ref{eq2}) simplifies as  
    \begin{align*}
    \left\{
    \begin{aligned}
        \dot{S} &= \varepsilon (1-S-H), \\ 
        \dot{H} &= -\gamma_H H, 
    \end{aligned}
    \right.
    \end{align*}
    meaning that $H$ decreases exponentially towards 0, while $S$ increases towards 1. Therefore the orbits converge towards the DFE $\mathbf{x_1}$.
    
    On the other hand, if $\mathcal{R}_0 > 1$ then clearly there exist some points in any neighborhood of the DFE such that $\dot{u}>0$, therefore in this case the DFE is unstable. 

    \medskip

    Finally, we remark that an alternative proof can be obtained using Lyapunov's Indirect Method. With this strategy, we are only able to prove the local asymptotic stability of the DFE, but we present it since it will be useful later. The Jacobian matrix of (\ref{eq2}) computed in the DFE $\mathbf{x_1}$ is 
    \begin{equation*}
        J|_{\mathbf{x}=\mathbf{x_1}} = 
        \begin{bmatrix}
            -\varepsilon & -\beta-\varepsilon & -\beta-\varepsilon & -\varepsilon \\ 
            0 & (1-\theta)\beta-\gamma_I & (1-\theta)\beta & 0 \\ 
            0 & \theta \beta & \theta \beta - \gamma_C & 0 \\ 
            0 & 0 & \gamma_C & -\gamma_H \\ 
        \end{bmatrix},
    \end{equation*}
    whose eigenvalues are 
    \begin{equation*}
        \lambda_1 = -\varepsilon < 0, \; \lambda_2=-\gamma_H < 0, \; \lambda_{3,4} = \frac{1}{2} \left(\beta-\gamma_I-\gamma_C \pm \sqrt{(\beta-\gamma_I-\gamma_C)^2-4 \gamma_I \gamma_C(1-\mathcal{R}_0)}\right). 
    \end{equation*} 
    
    Suppose that $\mathcal{R}_0>1$. Since 
    \begin{equation*}
        \sqrt{(\beta-\gamma_I-\gamma_C)^2-4 \gamma_I \gamma_C(1-\mathcal{R}_0)} > |\beta-\gamma_I-\gamma_C|
    \end{equation*}
    the eigenvalue $\lambda_3$ has positive real part, therefore Lyapunov's Indirect Method implies that the DFE is unstable. 

    Suppose now that $\mathcal{R}_0<1$. Lemma \ref{lemma1} implies that if $(\beta-\gamma_I-\gamma_C)^2-4 \gamma_I \gamma_C(1-\mathcal{R}_0) < 0$ the real parts of all eigenvalues are negative. Similarly, if $(\beta-\gamma_I-\gamma_C)^2-4 \gamma_I \gamma_C(1-\mathcal{R}_0) \ge 0$ then 
    \begin{equation*}
        \sqrt{(\beta-\gamma_I-\gamma_C)^2-4 \gamma_I \gamma_C(1-\mathcal{R}_0)} < |\beta-\gamma_I-\gamma_C|, 
    \end{equation*}
    which again implies that the real parts of all eigenvalues are negative. The local asymptotic stability of the DFE follows from Lyapunov's Indirect Method. 
\end{proof}

\begin{remark} \label{remstar}
    Since $\gamma_C \ge \gamma_I$, as $\theta$ grows the basic reproduction number $\mathcal{R}_0$ decreases. Indeed, as the disease becomes more dangerous, more individuals will be hospitalized and therefore they will not be able to spread the disease. In particular, it is possible that there exists a value $\theta^* \in (0,1)$ such that $\mathcal{R}_0 > 1 $ for all $\theta < \theta^*$ (which correspond to a growing epidemic), while $\mathcal{R}_0 < 1 $ for all $\theta > \theta^*$ (which correspond to the absence of an epidemic). 
\end{remark}

With an additional assumption on the rates $\beta$ and $\gamma_I$ we are able to prove the following global stability result for the DFE. 

\begin{proposition} \label{propexp}
    Suppose that $\beta < \gamma_I$, then the DFE is globally exponentially stable. 
\end{proposition}
\begin{proof}
    Notice that since $\beta < \gamma_I \le \gamma_C$, then $\mathcal{R}_0<1$. The number of infective individual converges exponentially towards zero, indeed 
    \begin{equation*}
    \begin{split}
        \dot{I}+\dot{C}&=\beta S (I+C) - \gamma_I I - \gamma_C C \\ 
        & \le \beta (I+C) - \gamma_I I - \gamma_I C \\
        & = \beta \left(1-\frac{\gamma_I}{\beta}\right)(I+C)
    \end{split}
    \end{equation*}
    and $1-\gamma_I/\beta < 0$. On the set $\{ I = C =0\}$, $H$ converges exponentially to 0 while $S$ to 1. This concludes the proof. 
\end{proof}

In particular, if $\gamma_I=\gamma_C=\gamma$ and $\mathcal{R}_0<1$, then the DFE is globally exponentially stable.

\subsection{Endemic equilibrium}

\noindent If $\mathcal{R}_0>1$ then system (\ref{eq2}) admits an endemic equilibrium (EE), i.e., an equilibrium in which $I, C, H > 0$. Denote such equilibrium with $\mathbf{x_2}=(S_2, I_2, C_2, H_2)$, calculations show that 
\begin{equation} \label{eqEE}
\begin{aligned}
    S_2 &= \frac{1}{\mathcal{R}_0}, \\ 
    I_2 &= \varepsilon\left(1-\frac{1}{\mathcal{R}_0}\right) \frac{1-\theta}{\theta} \frac{\gamma_C}{\gamma_I} \left(\frac{\gamma_C}{\theta}+\varepsilon\left(1+\frac{\gamma_C}{\gamma_H}+\frac{\gamma_C}{\gamma_I} \left(\frac{1}{\theta}-1\right)\right)\right)^{-1}, \\ 
    C_2 &= \varepsilon\left(1-\frac{1}{\mathcal{R}_0}\right)  \left(\frac{\gamma_C}{\theta}+\varepsilon\left(1+\frac{\gamma_C}{\gamma_H}+\frac{\gamma_C}{\gamma_I} \left(\frac{1}{\theta}-1\right)\right)\right)^{-1}, \\
    H_2 &= \varepsilon\left(1-\frac{1}{\mathcal{R}_0}\right) \frac{\gamma_C}{\gamma_H} \left(\frac{\gamma_C}{\theta}+\varepsilon\left(1+\frac{\gamma_C}{\gamma_H}+\frac{\gamma_C}{\gamma_I} \left(\frac{1}{\theta}-1\right)\right)\right)^{-1}.
\end{aligned}
\end{equation}
Notice that, since $\varepsilon \ll 1$, 
\begin{equation} \label{eqEEbis}
    \left(\frac{\gamma_C}{\theta}+\varepsilon\left(1+\frac{\gamma_C}{\gamma_H}+\frac{\gamma_C}{\gamma_I} \left(\frac{1}{\theta}-1\right)\right)\right)^{-1} = \frac{\theta}{\gamma_C} + \mathcal{O}(\varepsilon) \simeq \frac{\theta}{\gamma_C}. 
\end{equation} 
This implies that $I_2$, $C_2, H_2 \in \mathcal{O}(\varepsilon)$, therefore at the endemic equilibrium the vast majority of the population is composed by susceptible and recovered individuals.

\begin{theorem}
    If $\mathcal{R}_0>1$ the EE $\mathbf{x_2}$ is locally asymptotically stable. 
\end{theorem}
\begin{proof}
    We will neglect the $\mathcal{O}(\varepsilon^2)$ terms because $\varepsilon \ll 1$. Since $I_2+C_2 = \varepsilon (\mathcal{R}_0-1)/\beta$, the Jacobian matrix of (\ref{eq2}) computed in the EE $\mathbf{x_2}$ is 
    \begin{equation*}
        J|_{\mathbf{x}=\mathbf{x_2}} = 
        \begin{bmatrix}
            -\varepsilon(\mathcal{R}_0-1)-\varepsilon & -\frac{\beta}{\mathcal{R}_0}-\varepsilon & -\frac{\beta}{\mathcal{R}_0}-\varepsilon & -\varepsilon \\ 
            \varepsilon(1-\theta)(\mathcal{R}_0-1) & (1-\theta) \frac{\beta}{\mathcal{R}_0}-\gamma_I & (1-\theta) \frac{\beta}{\mathcal{R}_0} & 0 \\ 
            \varepsilon\theta(\mathcal{R}_0-1) & \theta \frac{\beta}{\mathcal{R}_0} & \theta \frac{\beta}{\mathcal{R}_0}-\gamma_C & 0 \\ 
            0 & 0 & \gamma_C & -\gamma_H \\ 
        \end{bmatrix}. 
    \end{equation*}
    Its characteristic polynomial is 
    \begin{equation*}
        p(\lambda)=(\lambda+\gamma_H) (\lambda^3+a\lambda^2+b\lambda+c) \eqqcolon (\lambda+\gamma_H) \; q(\lambda), 
    \end{equation*}
    where 
    \begin{equation*}
    \begin{split}
        a&=\frac{\theta \gamma_I^2 + (1-\theta)\gamma_C^2}{\theta \gamma_I + (1-\theta)\gamma_C} + \varepsilon \mathcal{R}_0, \\ 
        b&=\varepsilon \left( \mathcal{R}_0 (\gamma_I+\gamma_C) - \frac{\beta}{\mathcal{R}_0} \right), \\ 
        c&=\varepsilon \gamma_I \gamma_C (\mathcal{R}_0 - 1).  \\ 
    \end{split}
    \end{equation*}
    Obviously one eigenvalue is $\lambda_1=-\gamma_H < 0$. Since $a,b,c>0$, there are no positive real eigenvalues and there exists at least one negative real eigenvalue $\lambda_2<0$. We would like to show that also the remaining two eigenvalues $\lambda_{3,4}$ have a negative real part, from this the thesis would follow from Lyapunov's Indirect Method. If $\lambda_{3,4} \in \mathbb{R}$ then they are necessarily negative. Suppose that 
    \begin{equation*}
        \lambda_{3,4}=\alpha \pm i \beta 
    \end{equation*} 
    with $\alpha, \beta \in \mathbb{R}$. The Vieta's formula tells us that 
    \begin{equation*}
        \lambda_2 + 2 \alpha = -a, 
    \end{equation*}
    therefore 
    \begin{equation*}
        \alpha < 0 \Longleftrightarrow -a < \lambda_2 \Longleftrightarrow q(-a) < 0. 
    \end{equation*}
    Direct calculations show that 
    \begin{equation*}
        q(-a) = -\varepsilon \left(\mathcal{R}_0 \gamma_I + \mathcal{R}_0 \gamma_C -\frac{\beta}{\mathcal{R}_0}\right) \left(\frac{\theta \gamma_I^2 + (1-\theta)\gamma_C^2}{\theta \gamma_I + (1-\theta)\gamma_C}\right) + \varepsilon \gamma_I \gamma_C (\mathcal{R}_0-1) 
    \end{equation*}
    and that it is always negative. 
\end{proof}

\subsection{Bifurcation analysis}

\noindent For the bifurcation analysis we focus on the role of $\beta$, the rate at which the susceptible are infected by individuals characterized by a standard or critical course of the disease. Two situations must be distinguished: if $\mathcal{R}_0<1$ then only the DFE $\mathbf{x_1}$ exists and it is asymptotically stable; if $\mathcal{R}_0>1$ both the DFE $\mathbf{x_1}$ and the EE $\mathbf{x_2}$ exist, and the first one is unstable while the second one stable. Define 
\begin{equation*}
    \bar{\mathcal{R}}_0 \coloneqq \frac{\mathcal{R}_0}{\beta}=\frac{1-\theta}{\gamma_I}+\frac{\theta}{\gamma_C}, 
\end{equation*}
the two previous situations correspond to $\beta<1/\bar{\mathcal{R}}_0$ and $\beta>1/\bar{\mathcal{R}}_0$, respectively. Notice that the DFE $\mathbf{x_1}$ does not depend on $\beta$, while the EE $\mathbf{x_2}$ depends on it only through the basic reproduction number $\mathcal{R}_0$. Therefore define for every $X \in \{I,C,H\}$ 
\begin{equation*}
    p_{X_2} \coloneqq \frac{X_2}{1-\frac{1}{\mathcal{R}_0}},  
\end{equation*} 
which are independent of $\beta$. Figure \ref{f2} shows the bifurcation diagrams. The model exhibits a transcritical bifurcation at $\beta=1/\bar{\mathcal{R}}_0$ between the DFE and the EE. 

\begin{figure}
    \centering
    \begin{subfigure}{0.47\textwidth}
        \centering
        \includegraphics[width=\textwidth]{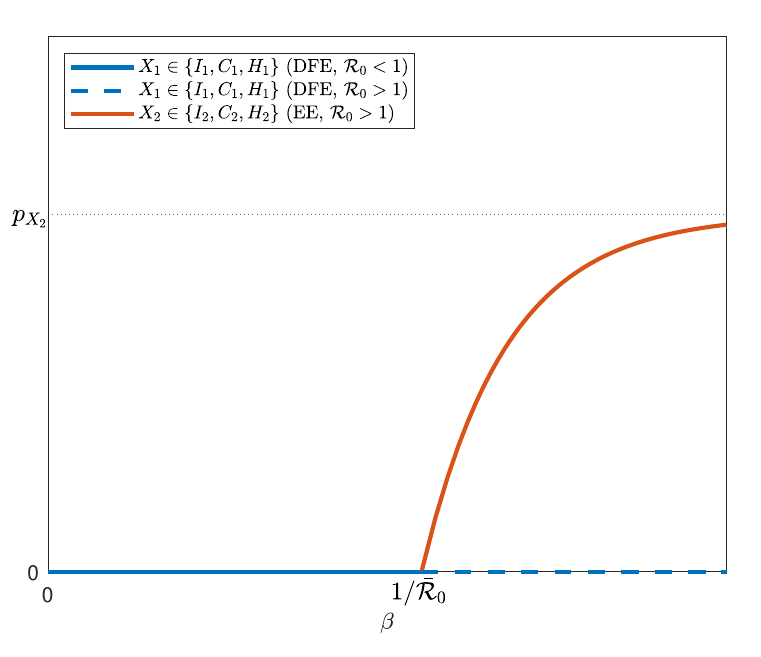}
        \caption{}
    \end{subfigure}
    \hspace{5mm}
    \begin{subfigure}{0.475\textwidth}
        \centering
        \includegraphics[width=\textwidth]{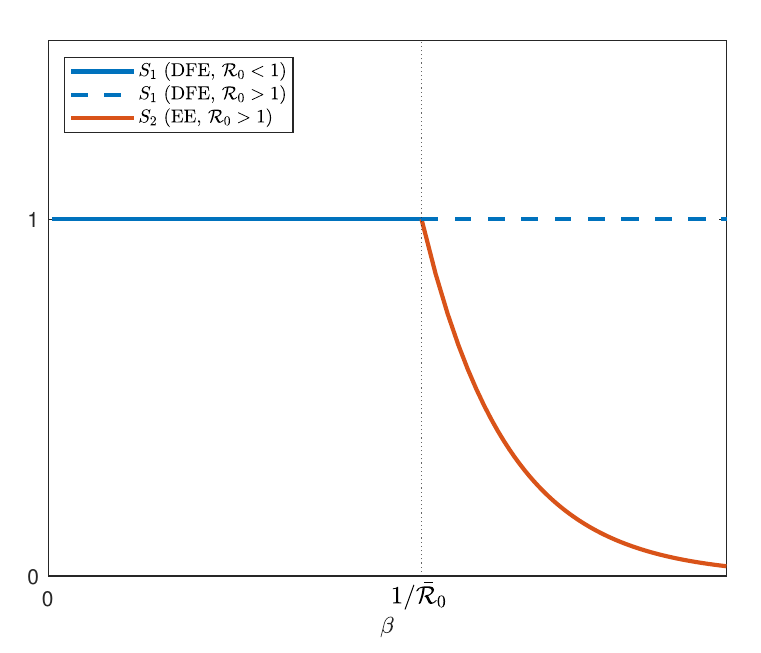}
        \caption{}
    \end{subfigure}
    \caption{Bifurcation analysis of system (\ref{eq2}), with a focus on the role of $\beta$, highlighting a transcritical bifurcation related to the exchange of stability between the DFE and the EE. Straight lines correspond to stable equilibria, while dashed lines to unstable ones. (A) Focus on the compartments $I$, $C$, and $H$, which all exhibit the same behavior. (B) Focus on the compartment $S$.} \label{f2}
\end{figure}

\section{Multiple timescale analysis} \label{section3}

\noindent In this section we analyze the two-timescale structure of system (\ref{eq2}). We restrict our analysis to the case $\mathcal{R}_0>1$, which corresponds to the spread of the epidemic. In order to highlight the two-timescale nature of system (\ref{eq2}), we write it as 
\begin{equation} \label{eq4}
    \dot{\mathbf{x}}=\mathbf{F_1}(\mathbf{x})+\varepsilon \; \mathbf{F_2}(\mathbf{x}), 
\end{equation}
where 
\begin{equation} \label{eq6}
    \mathbf{x}=
    \begin{bmatrix}
        S \\ I \\ C \\ H 
    \end{bmatrix}, \;
    \mathbf{F_1}(\mathbf{x})=
    \begin{bmatrix}
        -\beta S (I+C) \\ (1-\theta) \beta S(I+C) - \gamma_I I \\ 
        \theta \beta S (I+C) - \gamma_C C \\ \gamma_C C - \gamma_H H
    \end{bmatrix}, \; 
    \mathbf{F_2}(\mathbf{x})=
    \begin{bmatrix}
        1-S-I-C-H \\ 0 \\ 0 \\ 0
    \end{bmatrix}. 
\end{equation}

\subsection{Preliminaries on Geometric Singular Perturbation Theory} \label{GSPT} 

\noindent We start by providing a brief description of Geometric Singular Perturbation Theory (GSPT; \cite{2,23,40,0}), and in particular of the entry-exit function \cite{35}. Both of them will be fundamental to study the combination of the two different timescales of system (\ref{eq4}).

Consider the so-called fast-slow system in standard form 
\begin{align} \label{fss}
\left\{
\begin{aligned}
    \varepsilon \; \mathbf{x}' &= \mathbf{f}(\mathbf{x}, \mathbf{y}, \varepsilon), \\ 
    \mathbf{y}' &= \mathbf{g}(\mathbf{x}, \mathbf{y}, \varepsilon),
\end{aligned}
\right.
\end{align}
where $\mathbf{x}\in \mathbb{R}^n$, $\mathbf{y}\in \mathbb{R}^m$, $\mathbf{f}\colon \mathbb{R}^{n+m+1} \to \mathbb{R}^n$, $\mathbf{g}\colon \mathbb{R}^{n+m+1} \to \mathbb{R}^m$, $\mathbf{f}$ and $\mathbf{g}$ are $C^r$-smooth for some $r$ sufficiently large, and $0< \varepsilon \ll 1$ is a small parameter. The variable $\mathbf{x}$ is called fast variable while the variable $\mathbf{y}$ is called slow variable. System (\ref{fss}) is formulated on the slow timescale $\tau$, and the $'$ indicates the derivative with respect to $\tau$. By defining the fast time $t=\tau/\varepsilon$ it can be rewritten as
\begin{align} \label{fss_bis}
\left\{
\begin{aligned}
    \dot{\mathbf{x}} &= \mathbf{f}(\mathbf{x}, \mathbf{y}, \varepsilon), \\ 
    \dot{\mathbf{y}} &= \varepsilon \; \mathbf{g}(\mathbf{x}, \mathbf{y}, \varepsilon),
\end{aligned}
\right.
\end{align}
where the $^\cdot$ indicates the derivative with respect to $t$. The slow subsystem is defined by considering $\varepsilon=0$ in (\ref{fss}), which yields 
\begin{align} \label{fss_0}
\left\{
\begin{aligned}
    \mathbf{0} &= \mathbf{f}(\mathbf{x}, \mathbf{y}, 0), \\ 
    \mathbf{y}' &= \mathbf{g}(\mathbf{x}, \mathbf{y}, 0).
\end{aligned}
\right.
\end{align}
The slow flow defined by (\ref{fss_0}) is restricted to the critical manifold 
\begin{equation*}
    \mathcal{C}_0 \coloneqq \{(\mathbf{x}, \mathbf{y})\in \mathbb{R}^{n+m} \colon \; \mathbf{f}(\mathbf{x}, \mathbf{y})=\mathbf{0}\}, 
\end{equation*}
whose points are the equilibria of the fast subsystem 
\begin{align*} 
\left\{
\begin{aligned}
    \dot{\mathbf{x}} &= \mathbf{f}(\mathbf{x}, \mathbf{y}, 0), \\ 
    \dot{\mathbf{y}} &= \mathbf{0}.
\end{aligned}
\right.
\end{align*}
We provide now two definitions which will be fundamental for the analysis of our model \cite{40}. 

\begin{definition}
    A subset $\mathcal{M}_0 \subset \mathcal{C}_0$ is called normally hyperbolic if the $n \times n$ matrix $\textnormal{D}_\mathbf{x}\mathbf{f}(\mathbf{x}, \mathbf{y}, 0)$ of first partial derivatives with respect to the fast variables has no eigenvalues with zero real part for all $(\mathbf{x}, \mathbf{y}) \in \mathcal{M}_0$.
\end{definition}

\begin{definition}
    A normally hyperbolic subset $\mathcal{M}_0 \subset \mathcal{C}_0$ is called attracting if all eigenvalues of $\textnormal{D}_\mathbf{x}\mathbf{f}(\mathbf{x}, \mathbf{y}, 0)$ have negative real part for all $(\mathbf{x}, \mathbf{y}) \in \mathcal{M}_0$; similarly, $\mathcal{M}_0$ is called repelling if all eigenvalues have positive real part. If $\mathcal{M}_0$ is normally
    hyperbolic and neither attracting nor repelling, it is of saddle type.
\end{definition}

A basic result of GSPT is Fenichel's Theorem \cite[Theorem 3.1.4]{40} (see also \cite{23}). 

\begin{theorem}[Fenichel] \label{Fen}
    Consider a compact submanifold (possibly with boundary) $\mathcal{M}_0$ of the critical manifold $\mathcal{C}_0$. If $\mathcal{M}_0$ is normally hyperbolic, then for $\varepsilon>0$ sufficiently small, the following hold:
    \begin{enumerate}
        \item there exists a locally invariant manifold $\mathcal{M}_\varepsilon$, called slow manifold,  diffeomorphic to $\mathcal{M}_0$ (local invariance means that trajectories can enter or leave $\mathcal{M}_\varepsilon$ only through its boundaries); 
        \item $\mathcal{M}_\varepsilon$ is $\mathcal{O}(\varepsilon)$-close to $\mathcal{M}_0$; 
        \item the flow on $\mathcal{M}_\varepsilon$ converges to the slow flow as $\varepsilon \to 0$; 
        \item $\mathcal{M}_\varepsilon$ is $C^r$-smooth; 
        \item $\mathcal{M}_\varepsilon$ is normally hyperbolic and has the same stability properties with respect
        to the fast variables as $\mathcal{M}_0$ (attracting, repelling, or of saddle type); 
        \item $\mathcal{M}_\varepsilon$ is usually not unique but all the possible choices lie $\mathcal{O}(\exp(-D/\varepsilon))$-close to
        each other, for some $D > 0$; 
        \item the stable and unstable manifolds of $\mathcal{M}_\varepsilon$ are locally invariant and are also $\mathcal{O}(\varepsilon)$-close and diffeomorphic to the stable and unstable manifolds of $\mathcal{M}_0$. 
    \end{enumerate}
\end{theorem}

Note that point (6) of Theorem \ref{Fen} implies that the choice of the slow manifold $\mathcal{M}_\varepsilon$ does not change analytical and numerical results.

As we mentioned above, fast–slow systems like (\ref{fss}) and (\ref{fss_bis}) are said to be in standard form. In a more general context, it is possible
to analyze a fast–slow system in non-standard form given by \cite{0}
\begin{equation}\label{eq:nonstand}
    \dot{\mathbf{z}} = \mathbf{F}(\mathbf{z},\varepsilon), 
\end{equation}
with $\mathbf{z} \in \mathbb{R}^{n+m}$, $\mathbf{F}\colon \mathbb{R}^{n+m+1} \to \mathbb{R}^{n+m}$, and $\mathbf{F} \in C^r$, where the timescale separation is not explicit nor global. A system in the form (\ref{eq:nonstand}) is singularly perturbed if the set
\begin{equation*}
    \mathcal{C}_0:=\{ \mathbf{z}\in  \mathbb{R}^{n+m} \colon \; \mathbf{F}(\mathbf{z},0)=\mathbf{0}\}
\end{equation*}
is non-empty, nor consists of isolated singularities. In particular, system (\ref{eq4}) is in such non-standard form. However, sufficiently close to the critical manifold, we will be able to introduce a change of coordinates that brings our system in standard form, in order to apply the results presented in this subsection.

\begin{figure}[h!]
    \includegraphics[width=10cm]{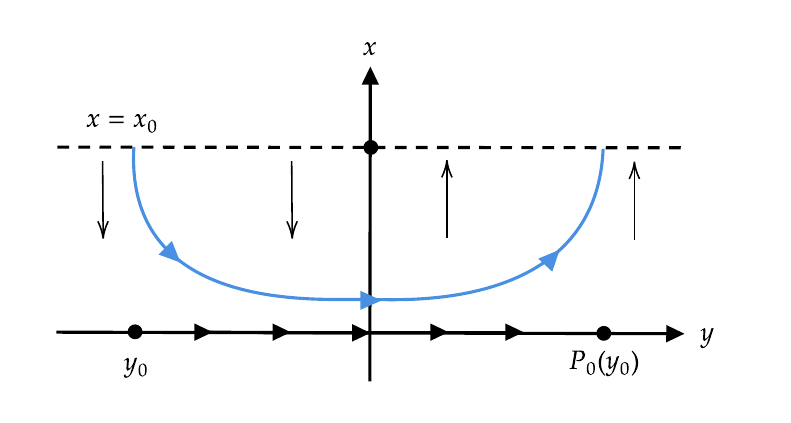}
    \caption{Visualization of the entry-exit map (\ref{e-e-f}) on the line $\{x=x_0\}$ with $x_0 \in \mathcal{O}(\varepsilon)$.} \label{f-1}
\end{figure}

Consider now the planar system 
\begin{align} \label{planar}
\left\{
\begin{aligned}
    \dot{x} &= x \; f(x,y,\varepsilon), \\ 
    \dot{y} &= \varepsilon \; g(x,y,\varepsilon),
\end{aligned}
\right.
\end{align}
with $(x,y) \in \mathbb{R}^2$, $g(0,y,0)>0$, and $\textnormal{sign}(f(0,y,0))=\textnormal{sign}(y)$. Notice that for $\varepsilon=0$ the $y$-axis consists of attracting/repelling equilibria if $y$ is negative/positive, respectively. Consider an orbit starting at $(x_0,y_0)$ for $x_0 \in \mathcal{O}(\varepsilon)$ and $y_0<0$ (see Figure \ref{f-1}). Intuitively, we expect that it is attracted to the $y$-axis as long as $y<0$ and that it will be then repelled away when $y>0$. Note that, since $g(0,y,0)>0$, we expect the $y$-coordinate of the orbit to grow during this process. However, since the $y$-axis is not normally hyperbolic, Fenichel's Theorem \ref{Fen} cannot explain this behavior since it could be applied only to a submanifold $\mathcal{M}_0$ of such axis away from the origin. On the other hand, the entry-exit function \cite{35} gives, in the form of a Poincaré map, an estimate of the behavior of such orbits near the origin. Consider the horizontal line $\{x=x_0\}$, the orbit of Figure \ref{f-1} re-intersects that line for $y=P_0(y_0) + \mathcal{O}(\varepsilon)$, where $P_0(y_0)$ is defined implicitly as the non-trivial solution to \cite{34} 
\begin{equation} \label{e-e-f}
    \int_{y_0}^{P_0(y_0)} \frac{f(0,y,0)}{g(0,y,0)}dy=0.
\end{equation}
It is important to remark that (\ref{e-e-f}) provides an approximation of the exit point, indeed it describes the orbits close to the $y$-axis through the flow on the $y$-axis itself (such description is only valid for $\varepsilon$ sufficiently small, see point (3) of Fenichel's Theorem \ref{Fen}). Let $dy=g(0,y,0) \; d\tau$, since the function $g$ describes the growth of the $y$-coordinate, one can transform (\ref{e-e-f}) into an integral equation which provides the (slow) exit time $\tau_E$: 
\begin{equation} \label{eqtime}
    \int_0^{\tau_E} f(0,y(\tau),0) \; d\tau = 0. 
\end{equation}
On the $y$-axis, the eigenvalues of the Jacobian of the fast subsystem of (\ref{planar}) are $\lambda_1=f(0,y,0)$ and $\lambda_2=0$. The former is associated to the fast variable $x$, while the latter to the slow variable $y$ (and it is obviously equal to zero since we are looking at the fast flow). Therefore (\ref{eqtime}) is equivalent to 
\begin{equation} \label{eqtime_eig}
    \int_0^{\tau_E} \lambda_1|_{y(\tau)} \; d\tau = 0. 
\end{equation}
Notice that $\lambda_1(y)<0$ if $y<0$, while $\lambda_1(y)>0$ if $y>0$. Indeed it describes the change of stability of the $y$-axis.  

\subsection{Fast formulation} \label{subfast}

\noindent Setting $\varepsilon=0$, we obtain the fast subsystem of (\ref{eq4}) 
\begin{align} \label{eq5}
\left\{
\begin{aligned}
    \dot{S} &= -\beta S (I+C), \\ 
    \dot{I} &= (1-\theta)\beta S (I+C) - \gamma_I I, \\ 
    \dot{C} &= \theta \beta S(I+C) - \gamma_C C, \\ 
    \dot{H} &= \gamma_C C - \gamma_H H.
\end{aligned}
\right.
\end{align}
The critical manifold $\mathcal{C}_0$ of (\ref{eq4}) is defined as the set of the equilibria of (\ref{eq5}), namely 
\begin{equation*}
    \mathcal{C}_0 \coloneqq \{(S,I,C,H) \in \mathbb{R}^4 \colon \; I=C=H=0\}. 
\end{equation*}
Notice that the first three equations of (\ref{eq5}) do not depend on $H$, which can therefore be written as 
\begin{equation*}
    H=H(t)=\left(H_0 + \gamma_C \int_0^t \exp(\gamma_H s) C(s) \; ds\right) \exp(-\gamma_H t), 
\end{equation*}
where $H_0=H(0)$. 

Let $X \in \{S,I,C,H\}$, denote with $X_0$ its initial condition, i.e., at the beginning of the fast flow. Moreover, define $X_\infty$ as 
\begin{equation*}
    X_\infty \coloneqq \lim_{t \to \infty} X(t), 
\end{equation*}
when this limit exists (under the flow of (\ref{eq5})). 

\begin{proposition}
    Suppose that $\mathcal{R}_0>1$. The trajectories of system (\ref{eq5}) converge to $\mathcal{C}_0$ as $t \to \infty$. 
\end{proposition}
\begin{proof}
    Notice that the solutions of the fast system (\ref{eq5}) evolve inside $\Delta$ given by (\ref{eq3}). Since $\dot{S} \le 0$ there exists $S_\infty \in [0, S_0]$. Moreover, $\dot{S}+\dot{I}, \dot{S}+\dot{C} \le 0$, therefore there exist also $I_\infty, C_\infty \ge 0$. Integrating $\dot{S}+\dot{I}$ and $\dot{S}+\dot{C}$, we respectively obtain 
    \begin{equation*}
    \begin{split}
        - \infty & < S_\infty + I_\infty - S_0 - I_0 = \int_0^\infty \left(\dot{S}(t)+\dot{I}(t)\right) \; dt \\ 
        & = -\theta \int_0^\infty S(t) \left(I(t)+C(t)\right) dt - \gamma_I \int_0^\infty I(t) \; dt < 0
    \end{split}
    \end{equation*}
    and 
    \begin{equation*}
    \begin{split}
        - \infty & < S_\infty + C_\infty - S_0 - I_0 = \int_0^\infty \left(\dot{S}(t)+\dot{C}(t)\right) \; dt \\ 
        & = -(1-\theta) \int_0^\infty S(t) \left(I(t)+C(t)\right) dt - \gamma_C \int_0^\infty C(t) \; dt < 0, 
    \end{split}
    \end{equation*}
    therefore $I_\infty=C_\infty=0$. Obviously there exists also $H_\infty=0$. 
\end{proof} 

\begin{lemma} \label{lemma2}
    Suppose that $\mathcal{R}_0>1$. Consider the $I$, $C$, and $H$ equations (see $\mathbf{F_1}$ in (\ref{eq6})) 
    \begin{equation*} 
    \begin{bmatrix}
        \dot{I} \\ \dot{C} \\ \dot{H} 
    \end{bmatrix}
    =
    \begin{bmatrix}
        (1-\theta) \beta S-\gamma_I & (1-\theta)\beta S  & 0 \\  
        \theta \beta S & \theta \beta S - \gamma_C & 0 \\ 
        0 & \gamma_C & -\gamma_H
    \end{bmatrix}
    \begin{bmatrix}
       I \\ C \\ H
    \end{bmatrix}. 
    \end{equation*}
    The eigenvalues of the matrix above are 
    \begin{equation} \label{eq7}
    \begin{split}
        &\lambda_1=-\gamma_H<0, \\ 
        &\lambda_{2} = \frac{1}{2} \left(  \beta S - \gamma_I - \gamma_C - \sqrt{           
        (\beta S - (\gamma_C - \gamma_I))^2 + 4(1-\theta) \beta (\gamma_C - \gamma_I) S} \right), \\ 
        &\lambda_{3} = \frac{1}{2} \left(  \beta S - \gamma_I - \gamma_C + \sqrt{           
        (\beta S - (\gamma_C - \gamma_I))^2 + 4(1-\theta) \beta (\gamma_C - \gamma_I) S} \right).
    \end{split}
    \end{equation}
    In particular, $\lambda_2 < 0 $ for all $S$, while 
    \begin{align} \label{eq8} 
        \begin{cases}
            \lambda_3 < 0 \quad & \textnormal{if} \; S<\frac{1}{\mathcal{R}_0}, \\ 
            \lambda_3 = 0 & \textnormal{if} \; S=\frac{1}{\mathcal{R}_0}, \\ 
            \lambda_3 > 0 & \textnormal{if} \; S>\frac{1}{\mathcal{R}_0}, \\ 
        \end{cases} 
    \end{align}
    and it is a strictly increasing function of $S$. 
\end{lemma}
\begin{proof}
    Direct calculations allow to deduce the expressions (\ref{eq7}). Notice that all three eigenvalues are real since the expression inside the square roots is positive, moreover 
    \begin{equation} \label{eq_equiv}
        \sqrt{           
        (\beta S - (\gamma_C - \gamma_I))^2 + 4(1-\theta) \beta (\gamma_C - \gamma_I) S} = \sqrt{(\beta S -\gamma_I - \gamma_C)^2 - 4 S \gamma_I \gamma_C \left(\frac{1}{S}-\mathcal{R}_0\right)}. 
    \end{equation} 

    In order to study the sign of $\lambda_2$, several cases must be distinguished. If $\beta < \gamma_I + \gamma_C$, then $\lambda_2$ is clearly negative. Suppose that $\beta > \gamma_I + \gamma_C$ and fix $S^*$. If $\beta S^* - \gamma_I - \gamma_C \le 0$ then $\lambda_2|_{S=S^*}<0$. Otherwise, if $S^* > (\gamma_I + \gamma_C)/\beta$, (\ref{eq_equiv}) implies that $\lambda_2|_{S=S^*}<0$ if and only if $S^* > 1/\mathcal{R}_0$. However $(\gamma_I + \gamma_C)/\beta > 1/\mathcal{R}_0$, therefore $\lambda_2 < 0$ for all $S$.

    Since $\beta/\mathcal{R}_0 - \gamma_I - \gamma_C <0$, from (\ref{eq_equiv}) it follows that $\lambda_3 |_{S=1/\mathcal{R}_0}=0$. Differentiating $\lambda_3$ (\ref{eq7}) with respect to $S$, simple but long calculations show that it is a strictly increasing function of $S$, therefore (\ref{eq8}) follows. 
\end{proof}

From Lemma \ref{lemma2} it follows that the critical manifold $\mathcal{C}_0$ is attracting for $S < 1/\mathcal{R}_0$, while it is of saddle type for $S > 1/\mathcal{R}_0$. Recall that $S_2=1/\mathcal{R}_0$. Therefore, in general, we expect that $S_\infty \in (0, 1/\mathcal{R}_0)$. However, pathological initial conditions may yield a different outcome. Indeed, notice that there are always two negative eigenvalues, therefore the critical manifold $\mathcal{C}_0$ is locally attractive along the corresponding eigenspaces. Let us start by considering $\lambda_1=-\gamma_H<0$, whose eigenspace is $E(-\gamma_H)=\textnormal{span}_\mathbb{R}\{(0,0,1)\}$. Set therefore $S_0>1/\mathcal{R}_0$, $I_0=0$, $C_0=0$, and $H_0>0$, it is clear that $(S_\infty, I_\infty, C_\infty, H_\infty)=(S_0, 0, 0, 0) \in \mathcal{C}_0$. However, any orbit lying on $E(-\gamma_H)$ is biologically irrelevant (see Subsection \ref{subsec1}), therefore we can assume that the flow evolves outside that eigenspace. Consider now the eigenspace $E(\lambda_2)$, recall that $\lambda_2<0$. Since $E(\lambda_2)$ is spanned by a vector with components of different signs, it does not lie within $\Delta$ (\ref{eq3}). This can easily be seen if $\gamma_I=\gamma_C=\gamma$, indeed in this case $\lambda_2=-\gamma$ and $E(\lambda_2)=\textnormal{span}_\mathbb{R}\{(-1,1,\gamma/(\gamma_H-\gamma))\}$. In conclusion, if the initial conditions lie inside $\Delta$ (\ref{eq3}) and $I_0+C_0>0$, then $S_\infty \in (0, 1/\mathcal{R}_0)$ (in particular, this is true if we follow Subsection \ref{subsec1}). Finally, supposing $\gamma_I=\gamma_C$, we can find an expression for $S_\infty$ and prove this fact in a different way. 

\begin{proposition}
    Assume that $\mathcal{R}_0>1$ and that $\gamma_I=\gamma_C=\gamma$, then 
    \begin{equation*}
        \Gamma(S,I,C)=\log{S} - \mathcal{R}_0 (S+I+C) 
    \end{equation*}
    is a constant of motion for system (\ref{eq5}). Moreover, if $I_0+C_0>0$ then $S_\infty \in (0, 1/\mathcal{R}_0)$. 
\end{proposition}
\begin{proof}
    By direct derivation with respect to time, we see that 
    \begin{equation*}
        \dot{\Gamma}(S,I,C)=-\beta(I+C) + \frac{\beta}{\gamma} (\gamma I +\gamma C) = 0. 
    \end{equation*}
    Therefore 
    \begin{equation*}
        \log{\frac{S_\infty}{S_0}} = \mathcal{R}_0 (S_\infty -S_0-I_0-C_0).  
    \end{equation*}
    Define over $(0, S_0)$ the function $h(x)=\log{\frac{x}{S_0}} - \mathcal{R}_0 (x -S_0-I_0-C_0)$, notice that $\lim_{x \to 0^+} h(x) = -\infty$ while $h(S_0)>0$. Since $\frac{dh}{dx}(x)=1/x - \mathcal{R}_0$, $h$ increases in $(0, \min \{1/\mathcal{R}_0, S_0\})$ and decreases in $(\min \{1/\mathcal{R}_0, S_0\}, S_0)$, hence there exists a unique zero of $h$, which represents $S_\infty$, and it belongs to $(0, 1/\mathcal{R}_0)$. 
\end{proof}

Until now, we have studied the flow of (\ref{eq5}), the fast subsystem of (\ref{eq4}). Before trying to understand the relationship between the orbits of the two systems, we will focus on the slow flow occurring near the critical manifold $\mathcal{C}_0$.

\subsection{Slow formulation} \label{sotto3}

\noindent Consider (\ref{eq2}) and assume that a solution reached an $\mathcal{O}(\varepsilon^2)$-neighborhood of the critical manifold $\mathcal{C}_0$, namely $I, C,H \in \mathcal{O}(\varepsilon^2)$. We rescale the compartments with sick individuals as $I=\varepsilon u$, $C=\varepsilon v$, and $H=\varepsilon w$. Moreover, apply a rescaling to the time variable, bringing the system to the slow timescale $\tau=\varepsilon t$: 
\begin{align} \label{eq9}
\left\{
\begin{aligned}
    S' &= -\beta S (u+v) + (1-S-\varepsilon u- \varepsilon v- \varepsilon w), \\ 
    \varepsilon u' &= (1-\theta)\beta S (u+v) - \gamma_I u, \\ 
    \varepsilon v' &= \theta \beta S(u+v) - \gamma_C v, \\ 
    \varepsilon w' &= \gamma_C v - \gamma_H w, 
\end{aligned}
\right.
\end{align}
where the ' indicates the derivative with respect to the slow time $\tau$. Notice that system (\ref{eq9}) has the same structure of system (\ref{fss}) and that $I,C,H \in \mathcal{O}(\varepsilon^2)$ implies $u,v,w \in \mathcal{O}(\varepsilon)$. However, since the critical manifold $\mathcal{C}_0$ is not normally hyperbolic (see Lemma \ref{lemma2}), Fenichel's Theorem \ref{Fen} cannot describe the whole behavior of the orbits during the slow flow. If we look at system (\ref{eq9}) on the critical manifold $\mathcal{C}_0$, now determined by $u=v=w=0$, we obtain 
\begin{equation} \label{eq11}
    S' = 1-S. 
\end{equation}
Therefore point (3) of Fenichel's Theorem \ref{Fen} implies that $S$ grows exponentially towards $1$ on the slow timescale and as long as we are looking the orbits away from $S=S_2=1/\mathcal{R}_0$. Indeed notice that if $\varepsilon \to 0$, then $u,v,w \to 0$ and the evolution of $S$ (\ref{eq9}) converges to the evolution described by Eq. (\ref{eq11}).

\begin{remark} \label{rmk2}
    If the initial conditions lie on $\mathcal{C}_0$ then the orbit converges towards the DFE $\mathbf{x_1}$, independent of $\mathcal{R}_0$. This is in line with the proof of Theorem \ref{teo1} where we showed that, despite the instability of the DFE for $\mathcal{R}_0>1$, the Jacobian matrix has three negative eigenvalues, ensuring local attraction along the corresponding eigenspaces. Using the notation of the proof of Theorem \ref{teo1}: $E(\lambda_1)=\textnormal{span}_\mathbb{R}\{(1,0,0,0)\}$ contains $\mathcal{C}_0$, which is itself a trajectory; $E(\lambda_2)=\textnormal{span}_\mathbb{R}\{(1,0,0,\gamma_H/\varepsilon-1)\}$ corresponds to the biologically irrelevant case $I_0=C_0=0$; $E(\lambda_4)$ is associated with non-admissible initial conditions, as it is spanned by a vector with components of different signs. 
\end{remark}

\subsection{Unified formulation and delayed loss of stability}

\noindent Standard perturbation theory \cite[Corollary 3.1.7]{2} implies that an orbit of the perturbed system (\ref{eq4}), away from the critical manifold $\mathcal{C}_0$, follows $\mathcal{O}(\varepsilon)$-closely the orbit of the fast subsystem (\ref{eq5}), related to the same initial conditions, for $\mathcal{O}(1)$ times $t$. Therefore, intuitively, we expect two possible scenarios. 

The first possibility is that as the orbit of the fast subsystem approaches $S_\infty$, the orbit of the perturbed system stays close to it and enters in the slow flow. In this case, when $I, C,H \in \mathcal{O}(\varepsilon^2)$ the influence of $\mathbf{F_2}$ (\ref{eq6}) becomes very relevant and, in the perturbed flow (\ref{eq4}), $S$ increases gently. In particular, the orbit is attracted to $\mathcal{C}_0$ as long as $S<1/\mathcal{R}_0$ since the manifold is attracting, hence the orbit stays close to $\mathcal{C}_0$ during this time window. On the other hand, when $S>1/\mathcal{R}_0$ the manifold is of saddle type, hence, sooner or later, it will be repelled away (indeed, as already explained, a biologically relevant solution is affected by the unstable component of $\mathcal{C}_0$ for $S>1/\mathcal{R}_0$). This behavior is represented in Figure \ref{f3}(A). Note that not necessarily any $X\in\{I,C,H\}$ decreases for $S<S_2$ and increases for $S>S_2$ (for example, $\dot{H}<0$ if and only if $\gamma_H H > \gamma_C C$). Biologically, this behavior represents the slow increment in the number of susceptible individuals between two different waves of an epidemic. Later a second epidemic wave will begin. 

The second possibility is that the perturbed orbit moves away from the one of the fast subsystem before entering in a $\mathcal{O}(\varepsilon^2)$-neighborhood of $\mathcal{C}_0$ (since any $X\in\{I,C,H\}$ decreases at most exponentially fast, a time $t$ of order at least $|\log \varepsilon| \gg 1$ is required to enter in such neighborhood). In this case we expect the orbit to quickly converge towards the EE (see Figure \ref{f3}(B)). Biologically, this evolution represents an epidemic without a full second wave but that does not come to an end. 

Numerical simulations show that both situations are possible (see Figures \ref{f4} and \ref{f5}). Note that only in Figure \ref{f4} the orbit enters in the slow flow. In particular, those numerical experiments showed that if the initial conditions satisfy $I_0+C_0>0$, then the orbits converge to the EE\footnote[1]{We do not report the results of all the numerical simulations carried out, but the orbits always showed a behavior similar to that of Figures \ref{f4} or \ref{f5}.}, which is $\mathcal{O}(\varepsilon)$-close to the critical manifold. Suppose that an orbit of (\ref{eq4}) is in a $\mathcal{O}(\varepsilon)$-neighborhood of $\mathcal{C}_0$ for $S<S_2$. We would like to understand when it enters in a $\mathcal{O}(\varepsilon^2)$-neighborhood of $\mathcal{C}_0$ starting the slow flow and its behavior during this process. Notice that, since the orbits converge towards the EE $\mathbf{x_2}$, for long times they cannot enter in a $\mathcal{O}(\varepsilon^2)$-neighborhood of $\mathcal{C}_0$ because $\mathbf{x_2}$ is $\mathcal{O}(\varepsilon)$-close to $\mathcal{C}_0$. The fact that the orbits start the slow flow in a $\mathcal{O}(\varepsilon^2)$-neighborhood of $\mathcal{C}_0$ should ensure that they are sufficiently far from the locally asymptotically stable EE so that they are not immediately attracted to it.

\begin{figure}
    \centering
    \begin{subfigure}{0.485\textwidth}
        \centering
        \includegraphics[width=\textwidth]{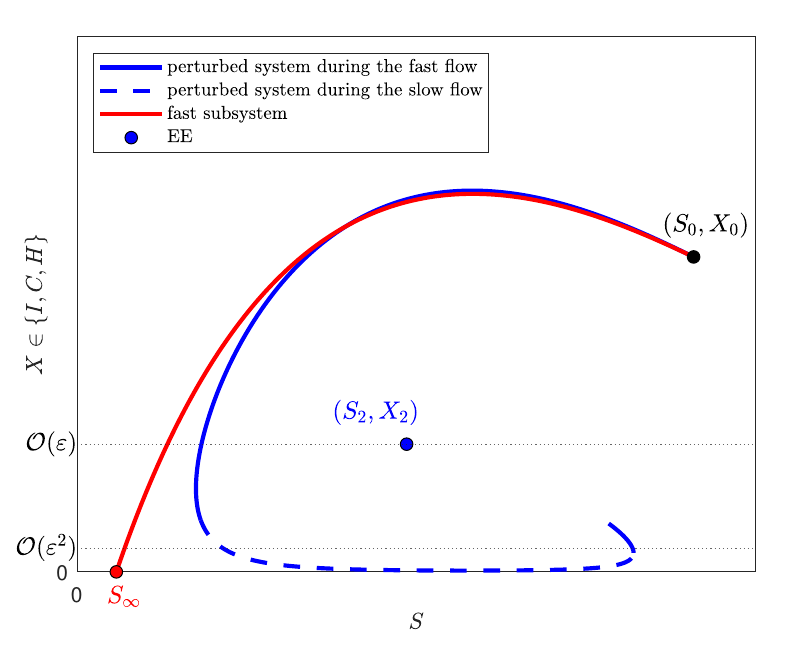}
        \caption{}
    \end{subfigure}
    \hspace{5mm}
    \begin{subfigure}{0.468\textwidth}
        \centering
        \includegraphics[width=\textwidth]{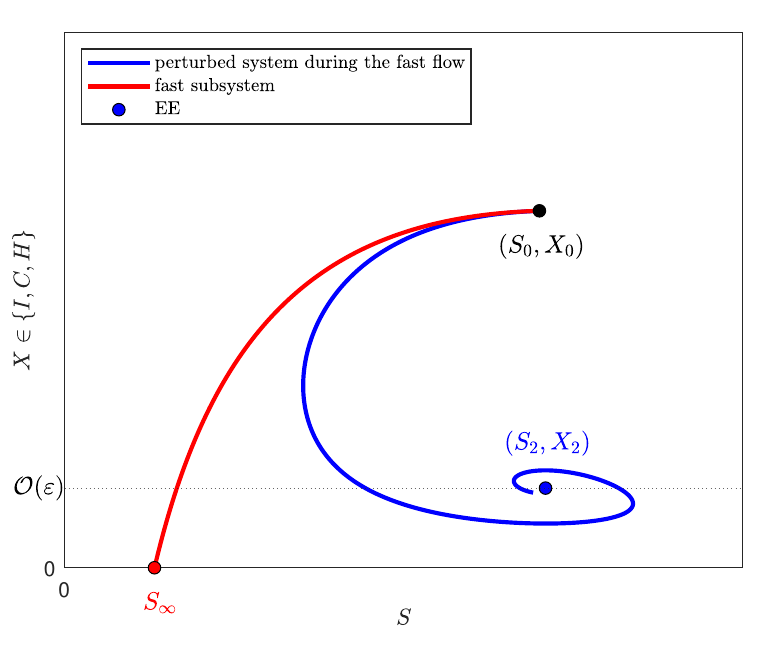}
        \caption{}
    \end{subfigure}
    \caption{Schematic representations of the orbits of the perturbed system (\ref{eq4}) and of its fast subsystem (\ref{eq5}), showing two possible behaviors. (A) The perturbed orbit stays $\mathcal{O}(\varepsilon)$-close to the orbit of (\ref{eq5}) until entering in a $\mathcal{O}(\varepsilon^2)$-neighborhood of $\mathcal{C}_0$. In this moment, the orbit enters in the slow flow, preparing a full second wave of the epidemic. (B) The perturbed orbit early moves away from the orbit of (\ref{eq5}) and it does not enter in a $\mathcal{O}(\varepsilon^2)$-neighborhood of $\mathcal{C}_0$. Instead, it quickly converges to the EE, hence there is not a full second wave of the epidemic.} \label{f3}
\end{figure}

\begin{figure}
    \centering
    \begin{subfigure}{0.47\textwidth}
        \centering
        \includegraphics[width=\textwidth]{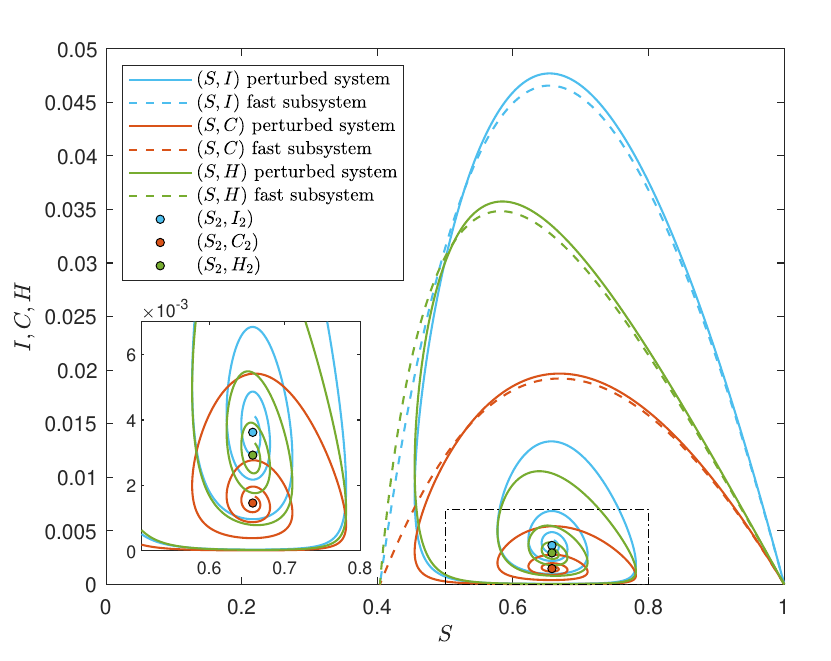}
        \caption{}
    \end{subfigure}
    \hspace{5mm}
    \begin{subfigure}{0.48\textwidth}
        \centering
        \includegraphics[width=\textwidth]{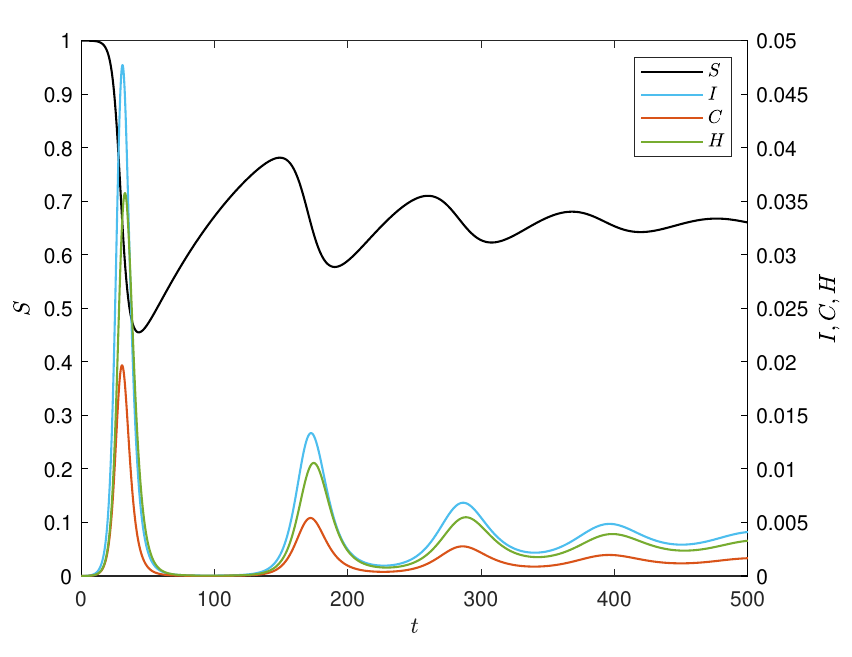}
        \caption{}
    \end{subfigure}
    \caption{Simulations of system (\ref{eq2}). The initial conditions have been chosen according to Subsection \ref{subsec1}, $I_0+C_0=10^{-5}$. After the first wave of the epidemic, the orbit enters in a $\mathcal{O}(\varepsilon^2)$-neighborhood of $\mathcal{C}_0$ starting the slow flow, note from the subgraph of (A) that during this process $I,C,H \in \mathcal{O}(\varepsilon^2)$. Later a second wave starts, but the orbit does not enter in a $\mathcal{O}(\varepsilon^2)$-neighborhood of $\mathcal{C}_0$ again. Indeed, in the subgraph of (A), it is clearly visible that $I$, $C$, and $H$ pass very close to $C_2 \not\in\mathcal{O}(\varepsilon^2)$. In (B) the two full waves of the epidemic are clearly visible. Notice that between them the values attained by $I$, $C$, and $H$ are much smaller than those attained after the second wave. For larger times, small perturbations around the EE create partial additional waves of the epidemic. We set $\beta=1$, $\theta=0.35$, $\gamma_I=0.6$, $\gamma_C=0.8$, $\gamma_H=0.4$, and $\varepsilon=0.01$.} \label{f4}
\end{figure}

\begin{figure}
    \centering
    \begin{subfigure}{0.47\textwidth}
        \centering
        \includegraphics[width=\textwidth]{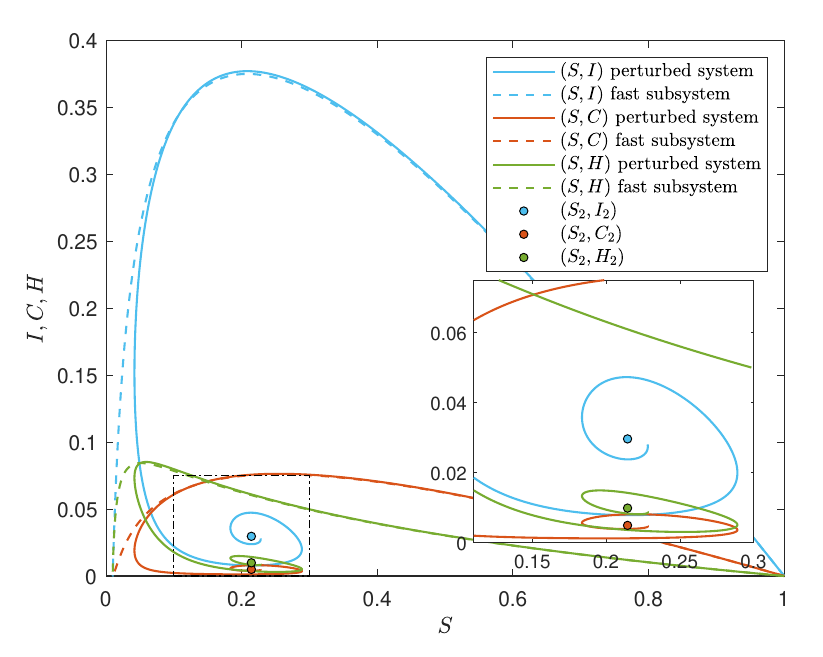}
        \caption{}
    \end{subfigure}
    \hspace{5mm}
    \begin{subfigure}{0.48\textwidth}
        \centering
        \includegraphics[width=\textwidth]{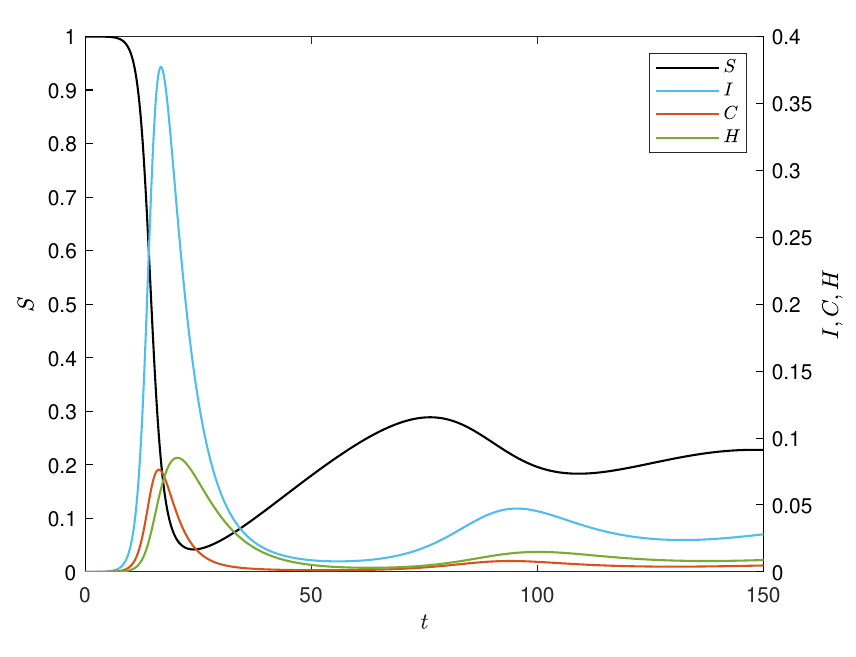}
        \caption{}
    \end{subfigure}
    \caption{Simulations of system (\ref{eq2}). The initial conditions have been chosen according to Subsection \ref{subsec1}, $I_0+C_0=10^{-5}$. After the first wave of the epidemic, the orbit does not enter in a $\mathcal{O}(\varepsilon^2)$-neighborhood of $\mathcal{C}_0$, indeed note from the subgraph of (A) that the minimum value attained by $I$ is larger than $C_2 \not\in \mathcal{O}(\varepsilon^2)$. Instead of starting the slow flow, the orbit quickly converges to the EE. In (B) it is visible that there is not a full second wave of the epidemic, there are only small perturbations caused by the convergence towards the EE. We set $\beta=1$, $\theta=0.2$, $\gamma_I=0.2$, $\gamma_C=0.3$, $\gamma_H=0.15$, and $\varepsilon=0.01$.} \label{f5}
\end{figure} 

Assume that an orbit of (\ref{eq4}) is in a $\mathcal{O}(\varepsilon)$-neighborhood of $\mathcal{C}_0$ and rescale the compartments with sick individuals as $I=\varepsilon u$, $C=\varepsilon v$, and $H=\varepsilon w$, bringing system (\ref{eq2}) into the standard structure (see Subsection \ref{GSPT}) 
\begin{align} \label{eq10}
\left\{
\begin{aligned}
    \dot{S}&=\varepsilon \; g(S,\mathbf{z}), \\ 
    \dot{\mathbf{z}}&=\mathbf{f}(S,\mathbf{z}), 
\end{aligned}
\right.
\end{align}
where
\begin{equation*} 
    \mathbf{z}=
    \begin{bmatrix}
        u \\ v \\ w
    \end{bmatrix}, \; 
    g(\mathbf{z}, S)=-\beta S (u+v) + (1-S-\varepsilon(u+v+w)), \;
    \mathbf{f}(S,\mathbf{z}) = 
    \begin{bmatrix}
        (1-\theta) \beta S (u+v) - \gamma_I u \\ 
        \theta \beta S(u+v) - \gamma_C v \\ 
        \gamma_C v - \gamma_H w 
    \end{bmatrix}. 
\end{equation*}
Notice that system (\ref{eq10}) can be equivalently written as 
\begin{align} \label{eq12}
\left\{
\begin{aligned}
    S' &= g(S,\mathbf{z}), \\ 
    \varepsilon \; \mathbf{z}' &=\mathbf{f}(S,\mathbf{z}), 
\end{aligned}
\right.
\end{align}
which depends on $\tau=\varepsilon t$, the previously introduced slow time. Clearly the critical manifold $\mathcal{C}_0$ is described by $\mathbf{z}=\mathbf{0}$. In order to understand whether the slow flow begins and to analyze the behavior of the orbits during this phase, we need Tikhonov's Theorem \cite[Theorem 4.1.2]{2}, which we report below in a version specific for our system. On the other hand, the entry-exit function will help us describe not only how the orbits behave as they pass near the non-hyperbolic point of $\mathcal{C}_0$, but also how they exit the slow flow to return in the fast flow. 

\begin{theorem} [Tikhonov] \label{teo3}
    Suppose that $\mathcal{R}_0>1$. Fix $S_0 \in (0, 1/\mathcal{R}_0)$ and one of its neighborhood $U_0 \subset (0,1/\mathcal{R}_0)$ such that $\lambda_3 \le -\lambda$ (\ref{eq7}) for all $S \in U_0$ and for some $\lambda>0$. Then there exists $\varepsilon_M, c_0, \dots, c_5>0$ such that the following properties hold for $0<\varepsilon < \varepsilon_M$: 
    \begin{enumerate}
        \item any solution of (\ref{eq12}) with initial conditions $(S_0, \mathbf{z_0})$ such that $||\mathbf{z_0}||<c_0$ satisfies 
        \begin{equation} \label{eq122}
            ||\mathbf{z}(\tau)|| \le c_1 \varepsilon + c_2 ||\mathbf{z_0}|| \exp\left(-c_3 \frac{\tau}{\varepsilon}\right) 
        \end{equation}
        for all $\tau \ge 0$ such that $S(s) \in U_0$ for $0 \le s \le \tau$; 
        \item let $S_{slow}(\tau)$ be the solution of (\ref{eq11}) with initial condition $S_{slow}(0)=S_0$, then 
        \begin{equation} \label{eq13}
            ||S(\tau)-S_{slow}(\tau)|| \le c_4 \varepsilon \exp(\tau) + c_5 ||\mathbf{z_0}|| \exp\left(-c_3 \frac{\tau}{\varepsilon}\right)
        \end{equation}
        for all $\tau \ge 0$ such that $S(s)\in U_0$ for $0 \le s \le \tau$. 
    \end{enumerate}
\end{theorem}

Relation (\ref{eq122}) implies that $\mathbf{z}$ reaches an $\mathcal{O}(\varepsilon)$-neighborhood of $\mathcal{C}_0$ in times $\tau$ of order $\varepsilon |\log{\varepsilon}|$, i.e., in times $t$ of order $|\log{\varepsilon}|$ (notice that $\mathbf{z} \in \mathcal{O}(\varepsilon)$ implies $I,C,H \in \mathcal{O}(\varepsilon^2)$). Moreover, the solution stays there as long as this manifold is attracting (i.e., as long as $S$ is lower than $S_2=1/\mathcal{R}_0$). Relation (\ref{eq13}) implies that $S(\tau) = S_{slow}(\tau) + \mathcal{O}(\varepsilon)$ for times $\tau$ of order between $\varepsilon |\log \varepsilon|$ and $1$, i.e., for times $t$ between $|\log \varepsilon|$ and $1/\varepsilon$, hence they both grow towards $1$. When $S>1/\mathcal{R}_0$ the critical manifold $\mathcal{C}_0$ becomes of saddle type, therefore, sooner or later, the orbit escapes from the slow flow and the system re-enters in the fast timescale, with small values of $I$, $C$, and $H$, like when an epidemic begins (indeed, now there will be a second wave). Now, the orbit can again show two different behaviors: it could directly converge to the EE, or it could enter in the slow flow for a second time. However, at some point, the orbit is attracted to the locally asymptotically stable EE $\mathbf{x_2}$, therefore, since $I_2,C_2,H_2 \not\in \mathcal{O}(\varepsilon^2)$, the previous process stops (for example, in Figure \ref{f4} the entry-exit phenomenon happens only once). Indeed, the hypothesis on the eigenvalue $\lambda_3$ of Tikhonov's Theorem \ref{teo3} is violated since $S$ is converging towards $S_2=1/\mathcal{R}_0$ and $\lambda_3=0$ if $S=1/\mathcal{R}_0$ (see Lemma \ref{lemma2}). Moreover, such theorem tells us that, in order to be attracted to $\mathcal{C}_0$, at some time the orbit must satisfy $|| \mathbf{z} || < c_0 $ when $S < 1/\mathcal{R}_0$ (however the value $c_0$ is not explicit). 

\begin{remark}
    The convergence to the EE of the orbit of Figure \ref{f4} is much slower than the convergence to the EE of the orbit of Figure \ref{f5}. This is due to the fact that only in the first case the orbit enters in a $\mathcal{O}(\varepsilon^2)$-neighborhood of $\mathcal{C}_0$, starting the slow dynamics of system (\ref{eq2}). Indeed, in the first case, after $t \simeq 40$ time the end of the first wave, $S$ reaches the value $1/\mathcal{R}_0$ (notice that $40$ is larger than $|\log \varepsilon| \simeq 4.61$ but smaller than $1/\varepsilon = 100)$. 
\end{remark}

Tikhonov's Theorem \ref{teo3} gave us information regarding the behavior of the orbits for $S<1/\mathcal{R}_0$ but it did not say anything about the orbits when $S>1/\mathcal{R}_0$. Since the critical manifold $\mathcal{C}_0$ is of saddle type for $S>1/\mathcal{R}_0$, we know that the orbits will be repelled away from $\mathcal{C}_0$, but we do not have precise information about their exact behavior or the moment at which this repulsion occurs. Indeed, the orbits may either exit the slow flow immediately or remain close to $\mathcal{C}_0$ until the accumulated contraction is balanced by the accumulated expansion, a phenomenon commonly referred to as delayed loss of stability \cite{15} (for example, this is precisely what occurs for the orbit of system (\ref{planar}) shown in Figure \ref{f-1}). During the slow flow, we are mainly interested in the compartment $S$ since its growth represents the slow increment in the number of susceptible individuals between two different waves of the  epidemic. In particular, we would like to obtain more information regarding the value attained by $S$ when the orbits leave the $\mathcal{O}(\varepsilon^2)$-neighborhood of $\mathcal{C}_0$ and the time needed to leave such neighborhood. To this end, we are going to employ the entry–exit function. 

We restrict this analysis to the case $\gamma_I=\gamma_C=\gamma$. Moreover, suppose that the initial conditions have been chosen according to Subsection \ref{subsec1}, which implies that system (\ref{eq2}) reduces to system (\ref{eqsimple}). Assume that the orbit entered in a $\mathcal{O}(\varepsilon^2)$-neighborhood of $\mathcal{C}_0$ and define $x=T/\varepsilon=u+v$ and $y=\mathcal{R}_0 S-1$, system (\ref{eq10}) simplifies to 
\begin{align} \label{system_xy}
\left\{
\begin{aligned}
    \dot{y} &= \varepsilon \left(\mathcal{R}_0-(y+1)\left(1+\beta x\right)\right) - \varepsilon^2 \mathcal{R}_0 (x+w) \eqqcolon \varepsilon \; g(x,y,\varepsilon) - \varepsilon^2 \; l(x,w,\varepsilon),\\
    \dot{x} &= \gamma x y \eqqcolon x \; f(x,y,\varepsilon),\\
    \dot{w} &= \gamma \theta x - \gamma_H w. 
\end{aligned}
\right.
\end{align}
If we neglect the $\mathcal{O}(\varepsilon^3)$-term $\varepsilon^2 l(x,w,\varepsilon)$, the first two equations of (\ref{system_xy}) have the exact same structure of the ones of (\ref{planar}). In particular, since $g(0,y,0)>0$ and $\textnormal{sign} (f(0,y,0))=\textnormal{sign}(y)$, the entry-exit function can be used to compute the exit point and the exit time\footnote[2]{Several articles have been devoted to the application of the entry-exit function to more complex systems \cite{37,38}. In \cite{39} a multi-dimensional version of the entry–exit function is provided, while in \cite{36} a result is achieved supposing that the stability of the critical manifold is determined by two intersecting eigenvalues.}. Consider the entry point $(S_0, T_0, H_0)$ such that $S_0<1/\mathcal{R}_0$ and $(T_0, H_0)$ is $\mathcal{O}(\varepsilon^2)$-close to $\mathcal{C}_0$ (hence $y_0=\mathcal{R}_0 S_0-1<0$ and $x_0 = T_0/\varepsilon \in \mathcal{O}(\varepsilon)$). The exit phenomenon happens when $T=T_0$ for $S=S_E>1/\mathcal{R}_0$. Eq. (\ref{e-e-f}) implies that an $\mathcal{O}(\varepsilon)$-approximation of $S_E$ is given by $S_E=(y_E+1)/\mathcal{R}_0$, where $y_E=P_0(y_0)$ satisfies
\begin{equation} \label{eq990}
    0 = \int_{y_0}^{y_E} \frac{f(0,y,0)}{g(0,y,0)} dy = \int_{y_0}^{y_E} \frac{\gamma \; y}{\mathcal{R}_0-y-1} dy, 
\end{equation}
namely 
\begin{equation} \label{eq999}
    y_E - y_0 +(\mathcal{R}_0 -1) \left( \log|\mathcal{R}_0-y_E-1|-\log|\mathcal{R}_0-y_0-1| \right)=0. 
\end{equation}
It is important to note that the term $\varepsilon^2 l(x,w,\varepsilon)$ can be neglected during this whole process since $\dot{w}>0$ if and only if $w < \gamma \theta x / \gamma_H$, hence $w \in \mathcal{O}(\varepsilon)$ as long as $x \in \mathcal{O}(\varepsilon)$. Eq. (\ref{eqtime}) ensures that an $\mathcal{O}(\varepsilon)$-approximation of the (slow) exit time $\tau_E$ satisfies 
\begin{equation} \label{eq998}
    0 = \int_0^{\tau_E} f(0,y(\tau),0) \; d\tau = \int_0^{\tau_E} \gamma \; y(\tau) \; d\tau = \int_0^{\tau_E} \left( \beta \; S(\tau)-\gamma \right) \; d\tau, 
\end{equation}
which argument of the last integral is indeed equal to the eigenvalue, associated to the fast variable $T$, describing the stability of the critical manifold $\mathcal{C}_0$ (see Eq. (\ref{eqtime_eig})). Since, on $\mathcal{C}_0$, $S$ grows towards $1$ (see Eq. (\ref{eq11})), (\ref{eq998}) admits a unique non-trivial solution, hence the same is true also for (\ref{eq990}). Finally, again Eq. (\ref{eq11}) implies that $\tau_E$ is implicitly given by 
\begin{equation} \label{eq1000}
    (\beta-\gamma)\tau_E - \beta (1-S_0)(1-\exp(-\tau_E))=0.
\end{equation}
In conclusion, such orbit of system (\ref{system_xy}) exhibits delayed loss of stability. 

Figure \ref{f7}(A) shows a comparison between the exit point computed numerically and with Eq. (\ref{eq999}), the maximum absolute error is approximately equal to $0.0011$, lower than $\varepsilon=0.01$. Figure \ref{f7}(B) shows a comparison between the slow exit time computed numerically and with Eq. (\ref{eq1000}), the maximum absolute error is approximately equal to $0.0033$, again lower than $\varepsilon=0.01$. Obviously, the fast exit time $t_E$ is given by $t_E=\tau_E/\varepsilon$, notice that it can be approximate up to an $\mathcal{O}(1)$ error (indeed, the maximum absolute value of this error related to the test presented in Figure \ref{f7}(B) is approximately $0.33$). 

\begin{figure}
    \centering
    \begin{subfigure}{0.485\textwidth}
        \centering
        \includegraphics[width=\textwidth]{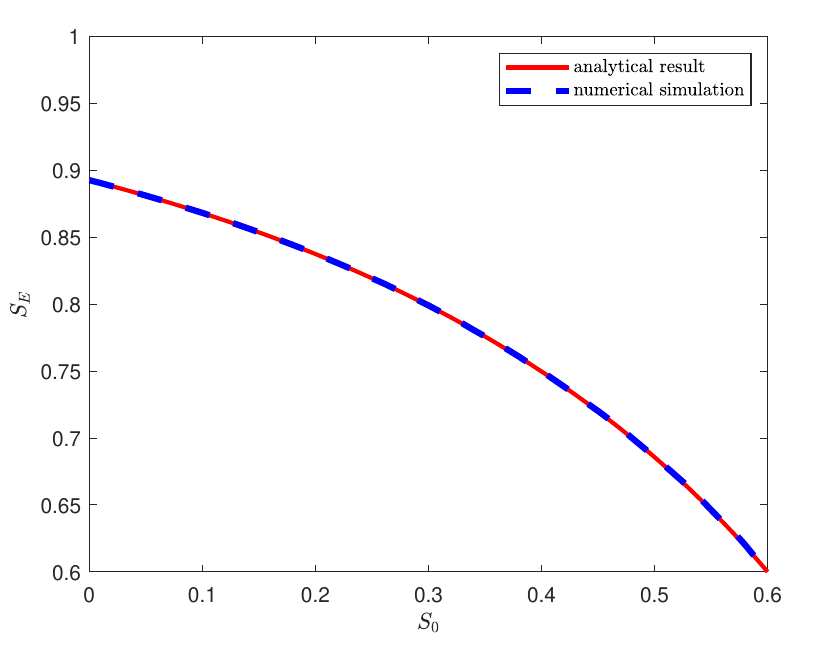}
        \caption{}
    \end{subfigure}
    \hspace{5mm}
    \begin{subfigure}{0.47\textwidth}
        \centering
        \includegraphics[width=\textwidth]{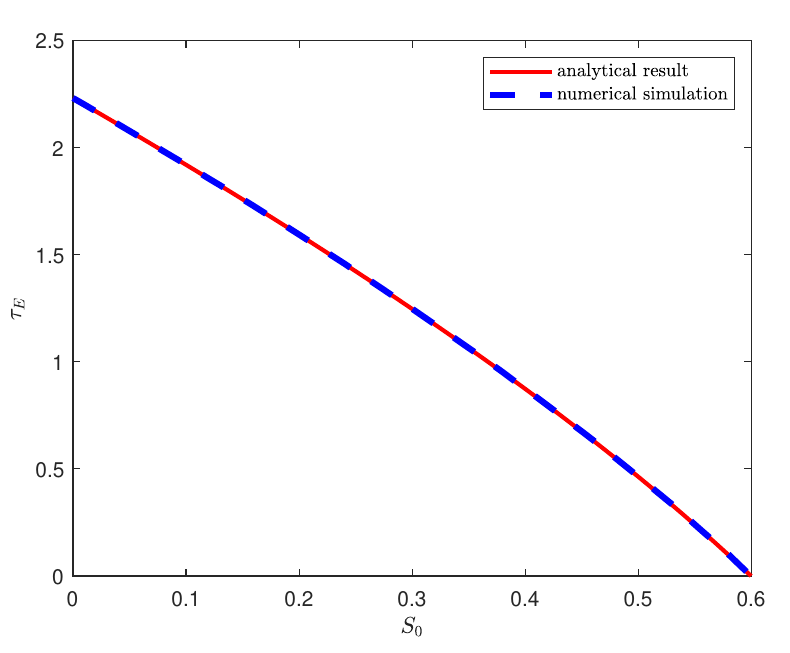}
        \caption{}
    \end{subfigure}
    \caption{Comparison between analytical predictions and numerical simulations illustrating the phenomenon of delayed loss of stability. (A) Exit points $S_E$ for entry points in the interval $S_0 \in (0, 1/\mathcal{R}_0)$, as predicted by Eq. (\ref{eq999}) and as obtained by direct integration of system (\ref{system_xy}) (recall that $S_E=(y_E+1)/\mathcal{R}_0$). (B) Slow exit times $\tau_E$ for entry points in the interval $S_0 \in (0, 1/\mathcal{R}_0)$, as predicted by Eq. (\ref{eq1000}) and as obtained by direct integration of system (\ref{system_xy}). The initial values for the fast variables were $(T_0, H_0)=(10^{-5}, 10^{-5})$.  We set $\beta=1$, $\theta=0.35$, $\gamma=0.6$, $\gamma_H=0.2$, and $\varepsilon=0.01$, hence $1/\mathcal{R}_0=0.6$.} \label{f7}
\end{figure} 

\section{Economic impact by disease severity} \label{section4}

\noindent In this section we are interested in understanding the influence of the probability that an infected individual will undergo a critical course of the disease $\theta$ on the economic impact of the epidemic. In particular, we would like to have a low number of hospitalized individuals since hospital care could have a big impact on the economy of a nation. 

Assuming that the (fixed) cost of hospital care for each individual is equal to $k>0$ per unit of time, at time $t$ the total cost $K=K(t)$ is given by 
\begin{equation} \label{eqk}
    K(t) = k \int_0^t H(s) \; ds.  
\end{equation}
Since $\theta=0$ obviously minimizes $K$ at any time $t$, rather than seeking the optimal value of $\theta$, we will instead determine its worst-case value. Specifically, we will find $\Tilde{\theta} \in (0,1]$ that maximizes $K$ at a chosen time $t$. Since studying the quantity $K$ is very complex, and it also depends on the initial conditions, we will divide our analysis into two parts, always supposing that $\mathcal{R}_0 > 1$ (which corresponds to the spread of the epidemic). We will start by analyzing what happens at the EE, since the orbits converge to it. Namely, we will initially focus on how $K(t)$ behaves for long times $t$. Finally, through numerical simulations, we will study the values attained by $K$ after the first wave of the epidemic. 

\subsection{Endemic equilibrium analysis} \label{EEanalysis}

\noindent For long times the trajectories converge to the EE, therefore the worst value for $\theta$ is the one that maximizes $H_2$ defined by (\ref{eqEE}). Recall that the EE exists if and only if $\mathcal{R}_0 > 1$. Due to (\ref{eqEEbis}), we are looking for the maximum value of 
\begin{equation} \label{eqftheta}
    f(\theta) \coloneqq \theta \left(1-\frac{1}{\mathcal{R}_0}\right), 
\end{equation}
for $\theta \in [0,1]$ such that $\mathcal{R}_0 = \mathcal{R}_0 (\theta) > 1$. Fixed the values of $\beta$,  $\gamma_I$, $\gamma_C$, $\gamma_H$, and $\varepsilon$, several situations must be distinguished.  

Suppose that $\gamma_I < \gamma_C < \beta$, in this case $\mathcal{R}_0 (\theta) > 1$ for all $\theta \in [0,1]$. First, notice that $f(0)=0$ and $f(1)>0$. A direct calculation shows that 
\begin{equation*} 
    \frac{df(\theta)}{d\theta} \le 0 \quad \Longleftrightarrow \quad 0 < \frac{\gamma_C}{\gamma_C - \gamma_I} \left(1 - \sqrt{\frac{\gamma_I}{\beta}}\right) \le \theta \le \frac{\gamma_C}{\gamma_C - \gamma_I} \left(1 + \sqrt{\frac{\gamma_I}{\beta}}\right). 
\end{equation*}
Since $\gamma_C \left(1 + \sqrt{\gamma_I/\beta}\right)/(\gamma_C - \gamma_I) > 1$, we are interested only in understanding when \linebreak $\gamma_C \left(1 - \sqrt{\gamma_I/\beta}\right)/(\gamma_C - \gamma_I) < 1$. Indeed, when this happens, the function $f$ reaches its peak for $\theta=\Tilde{\theta} \in (0,1)$. A simple calculation shows that the latter inequality is satisfied if and only if 
\begin{equation} \label{eqbeta}
    \beta < \frac{\gamma_C^2}{\gamma_I}, 
\end{equation}
which is in line with the assumption $\beta>\gamma_C>\gamma_I$. Hence, if (\ref{eqbeta}) is satisfied, then 
\begin{equation} \label{eqtilde}
    \Tilde{\theta} = \frac{\gamma_C}{\gamma_C - \gamma_I} \left(1 - \sqrt{\frac{\gamma_I}{\beta}}\right) \in (0,1), 
\end{equation}
otherwise $\Tilde{\theta}=1$. The maximum possible value of $\Tilde{\theta}$ is $1$, but what is its minimum possible value? It corresponds to $\beta \simeq \gamma_C$, which simplifies (\ref{eqtilde}) as $\sqrt{\gamma_C} / (\sqrt{\gamma_C} + \sqrt{\gamma_I})$. By choosing $\gamma_I$ in the best possible way, namely $\gamma_I = \gamma_C - \delta$, $0<\delta \ll \varepsilon \ll 1$, we get 
\begin{equation*}
    \Tilde{\theta} = \frac{\sqrt{\gamma_C}}{2\sqrt{\gamma_C}-\delta} > \frac{1}{2} \quad \textnormal{and} \quad \Tilde{\theta} \xrightarrow{\gamma_C \to \infty} {\frac{1}{2}}^+. 
\end{equation*}
In conclusion, supposing that $\gamma_I < \gamma_C < \beta$, if (\ref{eqbeta}) is satisfied then $\Tilde{\theta}$ is given by (\ref{eqtilde}) and it is always greater than $1/2$, otherwise $\Tilde{\theta}=1$. Notice that $\Tilde{\theta}$ cannot be arbitrarily close to $1/2$ since $\gamma_C$ cannot be arbitrarily large if we want our model to be biologically relevant. 

Suppose now that $\gamma_I < \beta < \gamma_C$, which implies that $\mathcal{R}_0(\theta) > 1$ and that the EE exists if and only if $\theta<\theta^*$, where $\theta^*$ is given by (see Remark \ref{remstar}) 
\begin{equation} \label{eqthetastar}
    \theta^* = \frac{\frac{1}{\gamma_I}-\frac{1}{\beta}}{\frac{1}{\gamma_I}-\frac{1}{\gamma_C}} \in (0,1). 
\end{equation} 
A simple calculation shows that we always have 
\begin{equation*}
    \frac{\gamma_C}{\gamma_C - \gamma_I} \left(1 - \sqrt{\frac{\gamma_I}{\beta}}\right) < \theta^*, 
\end{equation*}
therefore 
\begin{equation} \label{eqtildebis}
    \Tilde{\theta} = \frac{\gamma_C}{\gamma_C - \gamma_I} \left(1 - \sqrt{\frac{\gamma_I}{\beta}}\right) \in (0,\theta^*) \subset (0,1).  
\end{equation}
Notice that $f(0)=0$ and $f(\theta^*)=0$. Since $\theta^*$ can assume any value in $(0,1)$, $\Tilde{\theta}$ can be very close to zero, but what is its maximum possible value? Like before, we have to choose $\beta \simeq \gamma_C$, which simplifies (\ref{eqtildebis}) as $\sqrt{\gamma_C} / (\sqrt{\gamma_C} + \sqrt{\gamma_I})$, but now $\gamma_I$ must be chosen as small as possible.  However, since we must have $\varepsilon \ll \gamma_I$, the maximum value for $\Tilde{\theta}$ is close, but not arbitrarily close, to $1$. Notice that as $\beta \to \gamma_C^-$ we have $\theta^* \to 1^-$. In conclusion, supposing that $\gamma_I < \beta < \gamma_C$, $\Tilde{\theta}$ is given by (\ref{eqtildebis}) and it is always lower than $\theta^*$ defined by (\ref{eqthetastar}). Moreover, $\Tilde{\theta}$ can assume values arbitrarily close to $0$ and values close to $1$, but not arbitrarily close to it since $\gamma_I$ cannot be arbitrarily small if we want our model to be biologically relevant. 

\begin{remark}
    The only situation not considered was $\beta < \gamma_I < \gamma_C$. In this case, $\mathcal{R}_0(\theta) < 1$ for all $\theta \in [0,1]$, therefore the epidemic is about to end and there is not much point in talking about the costs. Moreover, since the DFE is globally exponentially stable (see Proposition \ref{propexp}), the quantity 
    \begin{equation*}
        K_\infty \coloneqq \lim_{t \to \infty} K(t) 
    \end{equation*}
    is well-defined (i.e., it is finite). 
\end{remark}

The previous analysis allowed to distinguish the following cases (notice that, obviously, the value of $\theta$ that minimizes the number of hospitalizations is also the value that minimizes the number of individuals characterized by a critical course of the disease): 
\begin{enumerate}
    \item if $\beta > \gamma_C^2/\gamma_I$, which implies that $\mathcal{R}_0>1$ for all $\theta$, then the epidemic spreads very rapidly and the hospitalizations cannot help to reduce the number of individuals characterized by a critical course of the disease, even if all infected individuals later become hospitalized and hence not able to spread the disease ($\Tilde{\theta}=1$); 
    \item if $\gamma_C < \beta < \gamma_C^2/\gamma_I$, which again implies that $\mathcal{R}_0>1$ for all $\theta$, then the epidemic spreads rapidly but not as fast as in the previous scenario, therefore the hospitalizations can help to reduce the number of individuals characterized by a critical course of the disease ($\Tilde{\theta}<1$ is given by (\ref{eqtilde})); 
    \item if $\gamma_I < \beta < \gamma_C$ then, if $\theta>\theta^*$ we have $\mathcal{R}_0 < 1$ meaning that there is no epidemic, whereas if $\theta<\theta^*$ we have $\mathcal{R}_0 > 1$ but spread of the epidemic is not fast and, like in the previous case, the hospitalizations can help to reduce the number of individuals characterized by a critical course of the disease ($\Tilde{\theta} < \theta^*$ are respectively given by (\ref{eqtildebis}) and (\ref{eqthetastar}));
    \item if $\beta < \gamma_I$ then $\mathcal{R}_0 < 1$ for all $\theta$, hence there is no epidemic. 
\end{enumerate}
Notice that if $\theta=1$ and $I_0=(1-\theta)(I_0+C_0)$ then $I \equiv 0$ and all individuals will undergo a critical course of the disease. Therefore the model would be equivalent to an SIRS model where all infected are quarantined before being recovered. 

\begin{remark}
    If $\gamma_I=\gamma_C=\gamma$ then $\mathcal{R}_0=\beta/\gamma$, which is independent of $\theta$. Supposing that $\beta>\gamma$, the value $\Tilde{\theta}$ that maximizes $f(\theta)$ is $\Tilde{\theta}=1$. Otherwise, if $\beta < \gamma$, there is no epidemic. These two situations correspond to cases (1) and (4), respectively. Cases (2) and (3) are obviously not possible if $\gamma_I=\gamma_C$.   
\end{remark}

\subsection{Numerical simulations}

\noindent During the first wave of the epidemic the orbits are not close to the EE, hence we cannot rely on the study of the function $f(\theta)$ defined by (\ref{eqftheta}) to understand the behavior of the total hospitalization costs $K$ given by (\ref{eqk}). We therefore performed several numerical tests to understand if the value of $\Tilde{\theta}$ obtained in Subsection \ref{EEanalysis} is a good approximation of the value of $\theta$ that maximizes $K$ after the first wave of the epidemic. Moreover, we would like to understand if the difference between the corresponding two values of $K$ is significant. We divided the numerical simulations into three groups, that correspond to cases (1), (2), and (3) described at the end of Subsection \ref{EEanalysis}. We did not consider case (4) since we would not have an epidemic. We chose the initial conditions following Subsection \ref{subsec1}. For each group we fixed the values of $\beta$, $\gamma_I$, $\gamma_C$, $\gamma_H$, and $\varepsilon$, while we varied the value of $\theta$. Since the duration of the first wave of the epidemic changed as the parameters varied, for each simulation we chose as final time $t_F$ the first value of $t$ such that $S(t)=S_2$, $I(t)<I_2$, and $C(t)<C_2$, i.e., when the first wave was already finished (see Figures \ref{f4}(A) and \ref{f5}(A)). 

\begin{figure}
    \centering
    \begin{subfigure}{0.47\textwidth}
        \centering
        \includegraphics[width=\textwidth]{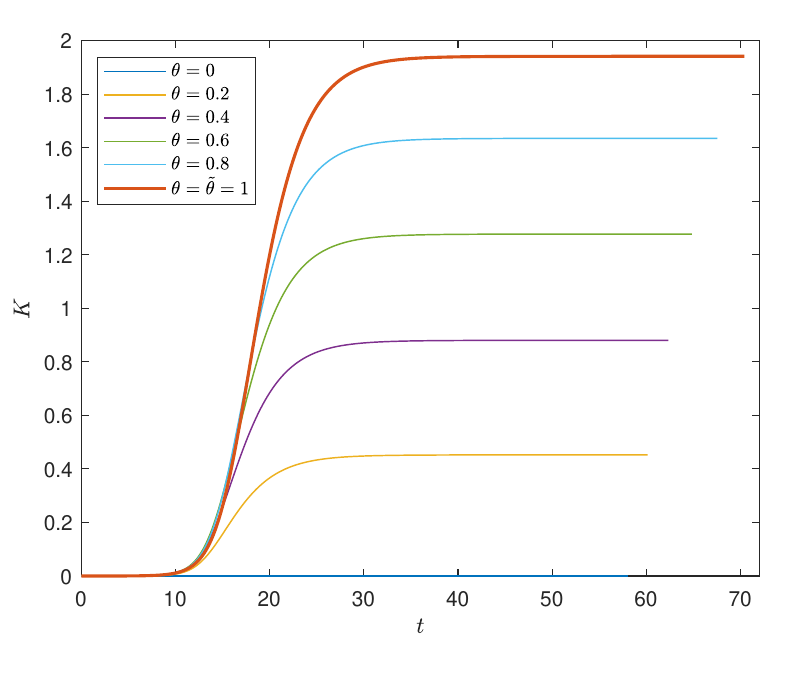}
        \caption{}
    \end{subfigure}
    \hspace{5mm} 
    \begin{subfigure}{0.48\textwidth}
        \centering
        \includegraphics[width=\textwidth]{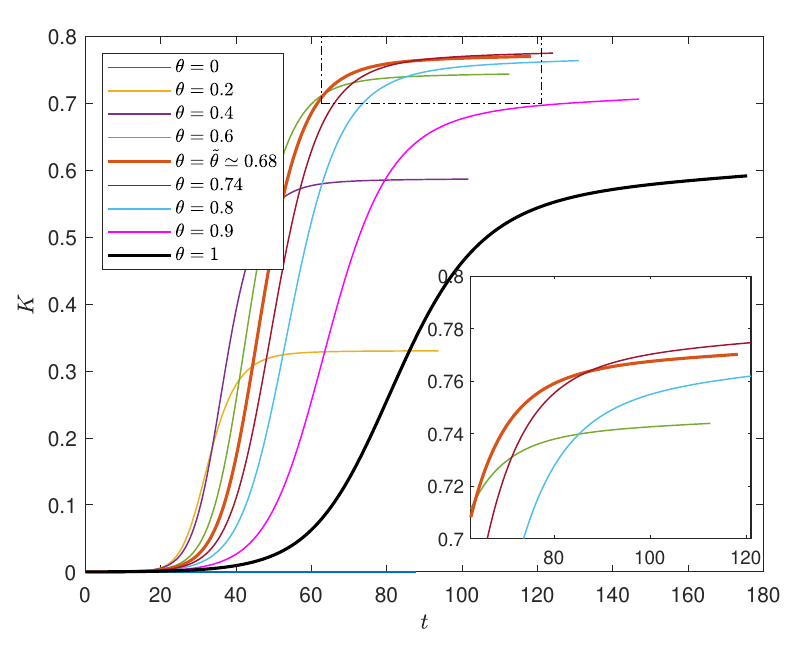}
        \caption{}
    \end{subfigure}
    \vspace{5mm}
    \begin{subfigure}{0.475\textwidth}
        \centering
        \includegraphics[width=\textwidth]{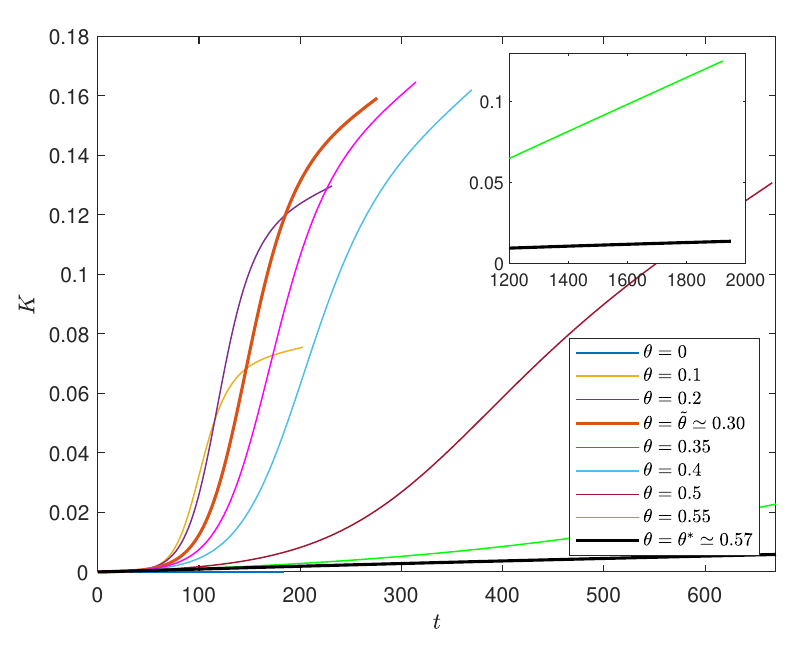}
        \caption{}
    \end{subfigure}
    \caption{Numerical simulations of the evolution of hospitalization costs $K$, as defined by Eq. (\ref{eqk}) with $k=1$, corresponding to cases (1), (2), and (3) described in Subsection \ref{EEanalysis}; each case represents a different epidemic scenario, ranging from the most aggressive (case (1)) to the least aggressive (case (3)). For each case, different values of disease severity $\theta$ are considered. The initial conditions have been chosen according to Subsection \ref{subsec1}, $I_0+C_0=10^{-5}$. (A) Case (1), $\beta=1.5$, $\gamma_I=0.6$, $\gamma_C=0.8$, $\gamma_H=0.4$, and $\varepsilon=0.01$. (B) Case (2), $\beta=1$, $\gamma_I=0.6$, $\gamma_C=0.9$, $\gamma_H=0.4$, and $\varepsilon=0.01$. (C) Case (3), $\beta=0.7$, $\gamma_I=0.6$, $\gamma_C=0.8$, $\gamma_H=0.4$, and $\varepsilon=0.01$.} \label{f6}
\end{figure}

Figure \ref{f6}(A) shows the results of the numerical simulations corresponding to case (1). Notice that the total cost $K$ increases significantly as $\theta$ increases (obviously $K \equiv 0$ if $\theta=0$). The maximum value of $K$ is related to $\theta=\Tilde{\theta}=1$, hence this analysis perfectly matches with the results obtained in Subsection \ref{EEanalysis}, which were valid for long times $t$. Since at the beginning of the epidemic wave the curves intersect with each other several times, it is not possible to tell what is the worst value of $\theta$ for these times $t$.  

Figure \ref{f6}(B) shows the results of the numerical simulations corresponding to case (2). Notice that the total cost $K$ initially increases as $\theta$ increases, reaching its peak after the first wave for $\theta \simeq 0.74$, then for larger values of $\theta$ it decreases. It is worth noticing that not only $0.74$ is not far from $\Tilde{\theta} \simeq 0.68$, but, moreover, the values of $K$ related to $\theta = 0.74$ and $\theta=\Tilde{\theta}$ are almost identical (approximately $0.78$ and $0.77$, respectively). Hence, under a practical point of view, $\Tilde{\theta}$ can be considered a valid approximation of the worst value of $\theta$. On the other hand, the costs $K$ related to $\theta=1$ are much lower than the previous two (approximately $0.59$). Like before, during the wave of the epidemic the curves intersect with each other several times. 

Finally, Figure \ref{f6}(C) shows the results of the numerical simulations corresponding to case (3). We restricted our analysis to $\theta \in [0, \theta^*]$ since for larger values of it we would get $\mathcal{R}_0 < 1$, which implies that there is no epidemic. Notice that if $\theta=\theta^*$, meaning that $\mathcal{R}_0=1$, then the total cost $K$ is very low. This is due to the fact that the orbit was attracted to the DFE, hence there was no epidemic (in this case, $t_F$ was not defined). Numerically, we see that the total cost $K$ initially increases as $\theta$ increases, reaching its peak after the first wave for $\theta \simeq 0.35$, then for larger values of $\theta$ it decreases. Again, $0.35$ is not far from $\Tilde{\theta} \simeq 0.30$, but, in particular, the values of $K$ related to $\theta = 0.35$ and $\theta=\Tilde{\theta}$ are very similar (approximately $0.160$ and $0.164$, respectively). Hence, under a practical point of view, also in this case $\Tilde{\theta}$ can be considered a valid approximation of the worst value of $\theta$. Finally, the costs related to values of $\theta$ close to $\theta^* \simeq 0.57$ are much lower than the previous two (for example, if $\theta=0.55$ then the costs are approximately $0.125$). Again, during the wave of the epidemic the curves intersect with each other several times. 

In conclusion, regardless of the case (1), (2), or (3), the value of $\theta$ that maximizes the hospitalizations costs after the first wave of the epidemic is close to the value $\Tilde{\theta}$ obtained analytically in Subsection \ref{EEanalysis}. In particular, in case (1), they seem to be identical (several other numerical tests, not presented here, supported this hypothesis). 

\begin{remark}
    Notice that, in general, as $\theta$ increases the final time $t_F$ increases. Indeed, if $\theta$ grows then  the basic reproduction number $\mathcal{R}_0$ decreases, therefore the epidemic spreads slower. 
\end{remark}

\subsection{Concluding remarks on $\gamma_I$ and $\gamma_C$}

\noindent Throughout this paper we assumed that $\gamma_I \le\gamma_C$. We made this choice because we supposed that severely infected individuals will require hospital treatment before standard infected individuals are recovered. Moreover, in the Introduction, we explained that our new model can also be interpreted as a scenario where only a fraction $\theta$ of infected individuals develops symptoms and self-quarantines. In this case, as soon as an individual realizes they are infected they should self-quarantine, so we expect this to happen faster than an asymptomatic individual recovers from the disease. But what happens if $\gamma_I > \gamma_C$? This scenario corresponds to a disease from which one can recover quickly, or develop symptoms but not immediately after being infected. All results of Sections \ref{section2} and \ref{section3} hold also in this case, in fact simply reverse $\gamma_I$ and $\gamma_C$ in the calculations. The only exception is Remark \ref{remstar}. Indeed, if $\gamma_I > \gamma_C$ then the basic reproduction number $\mathcal{R}_0$ increases as $\theta$ grows. This implies that, as the disease becomes more dangerous, the spreading speed of the epidemic increases. In particular, the function $f(\theta)$ defined by Eq. (\ref{eqftheta}) becomes an increasing function of $\theta$. Therefore the value of $\theta$ that maximizes the total hospitalization costs $K$, given by Eq. (\ref{eqk}), is always $\Tilde{\theta}=1$. In this scenario it is not possible to reduce the number of hospitalizations with hospitalizations themselves. Indeed, to do this we based ourselves on the fact that hospitalizations were quicker than the recovery times from a standard course of the disease. In such situation, we can only rely on a control strategy to reduce the spread of the disease and the number of critical cases (see, for example, \cite{27,19,18, 25, 33, 26}). 

\section{Conclusions} \label{section5}

\noindent We introduced an SIRS-type model to describe the evolution of contagions supposing that a fraction $\theta$ of infected individuals experiences a severe course of the disease, requiring hospitalization. During hospitalization, these individuals do not contribute to further infections. The model evolves on two different timescales, indeed we supposed that the loss rate of complete immunity to be much smaller than the other rates describing the evolution of the system. 

The basic reproduction number $\mathcal{R}_0$ depends on the parameters $\beta$, $\gamma_I$, $\gamma_C$, and $\theta$. Since we assumed that severely infected individuals require hospital treatment before standard infected individuals are recovered ($\gamma_I \le \gamma_C$), as $\theta$ increases $\mathcal{R}_0$ decreases. This implies that as the disease becomes more dangerous, the spread of the epidemic reduces thanks to hospitalizations. In particular, it is possible that there exists a value $\theta^*$ such that $\mathcal{R}_0>1$ for all $\theta<\theta^*$ (which correspond to a growing epidemic), while $\mathcal{R}_0 < 1$ for all $\theta > \theta^*$ (which correspond to the absence of an epidemic). This property of the basic reproduction number was of fundamental importance when we studied the total cost of the hospitalizations related to a growing epidemic as a function of $\theta$. For long times, we were able to compute the explicit value $\Tilde{\theta}$ that maximized those costs. One could expect that $\Tilde{\theta}=1$, meaning that all infected individuals undergo a critical course of the disease, and hence require hospital treatments. Interestingly, this is not true. We proved that if $\beta > \gamma_C^2/\gamma_I$, which corresponds to an epidemic spreading extremely quickly, then $\Tilde{\theta}=1$. However, if $\beta < \gamma_C^2/\gamma_I$, then $\Tilde{\theta}<1$. This highlights the interesting effect that a severe disease, by necessitating widespread hospitalization, could indirectly suppress transmission and, consequently, reduce hospitalizations. In order to study the total hospitalization costs after the first wave of the epidemic we performed several numerical tests. They all showed that the value of $\theta$ that maximized those costs is close to the value $\Tilde{\theta}$ obtained explicitly. In particular, the corresponding costs are almost identical under a practical point of view. 

From a mathematical point of view, our analysis relied mostly on Geometric Singular Perturbation Theory. When $\mathcal{R}_0>1$, the system can evolve on two different timescales: the fast one describing the waves of the epidemic, and the slow one describing the  slow increment in the number of susceptible individuals between two different waves of the epidemic. However, not always the orbits enter in the slow flow and, in these case, there is not a complete second wave of the epidemic, instead the orbits quickly converge towards the endemic equilibrium. We focused our mathematical analysis on the transition from one timescale to another. In particular, we exploited Fenichel's and Tikhonov's Theorems and the entry-exit function to study the behavior of the orbits during the slow flow.

Finally, we studied what happens if the main assumption of the paper on $\gamma_I$ and $\gamma_C$ is not satisfied, namely if $\gamma_I > \gamma_C$. In this case, we showed that the only difference is that $\mathcal{R}_0$ becomes an increasing function of $\theta$, meaning that we always have $\Tilde{\theta}=1$. In this case it is not possible to reduce the number of hospitalizations with hospitalizations themselves because they are no longer quicker than the recovery times from a standard course of the disease. However, we remarked that our model can also be interpreted as a scenario where only a fraction $\theta$ of infected individuals develops symptoms and self-quarantines. In this situation, it is probably reasonable to assume that symptomatic individuals self-quarantine before asymptomatic ones recover. 

Future developments involve the implementation of control strategies aimed at reducing infections and hospitalizations. For instance, mass vaccination campaigns could be considered for susceptible individuals, while quarantine measures might be applied to those experiencing a standard course of the disease. Specifically, it would be valuable to investigate the effects of these strategies in a scenario such that $\gamma_I > \gamma_C$. In this context, since both the spreading speed of the epidemic and the severity of the disease increase as $\theta$ grows, it would be interesting to determine whether an appropriate control strategy could mitigate these effects.

\bigskip
\bigskip
\noindent \textbf{Acknowledgments.} The author is grateful to Mattia Sensi for the insightful discussion which enriched the paper. This work was performed in the frame of activities sponsored by the University of Parma and by the Italian National Group of Mathematical Physics (GNFM-INdAM). The author also thanks the support of the project PRIN 2022 PNRR ”Mathematical Modelling for a Sustainable Circular Economy in Ecosystems” (project code P2022PSMT7, CUP D53D23018960001) funded by the European Union - NextGenerationEU, PNRR-M4C2-I 1.1, and by MUR-Italian Ministry of Universities and Research.

\bigskip
\nocite{*}
\bibliographystyle{plain}
\bibliography{Bibliography_EE}

\begin{thebibliography}{10}

\bibitem{27}
M.D. Ahmad, M.~Usman, A.~Khan, and M.~Imran.
\newblock Optimal control analysis of {E}bola disease with control strategies of quarantine and vaccination.
\newblock {\em Infect. Dis. Poverty}, 5:72, 2016.
\newblock \url{https://doi.org/10.1186/s40249-016-0161-6}.

\bibitem{9}
V.~Andreasen.
\newblock The effect of age-dependent host mortality on the dynamics of an endemic disease.
\newblock {\em Math. Biosci.}, 114:29--58, 1993.
\newblock \url{https://doi.org/10.1016/0025-5564(93)90041-8}.

\bibitem{22}
H.~Behncke.
\newblock Optimal control of deterministic epidemics.
\newblock {\em Optimal Control Appl. Methods}, 21:269--285, 2000.
\newblock \url{https://doi.org/10.1002/oca.678}.

\bibitem{2}
N.~Berglund.
\newblock Perturbation theory of dynamical systems.
\newblock {\em \texttt{arXiv:math/0111178}}, 2001.
\newblock \url{https://doi.org/10.48550/arXiv.math/0111178}.

\bibitem{19}
L.~Bolzoni, E.~Bonacini, R.~Della~Marca, and M.~Groppi.
\newblock Optimal control of epidemic size and duration with limited resources.
\newblock {\em Math. Biosci.}, 315:108232, 2019.
\newblock \url{https://doi.org/10.1016/j.mbs.2019.108232}.

\bibitem{18}
L.~Bolzoni, E.~Bonacini, C.~Soresina, and M.~Groppi.
\newblock Time-optimal control strategies in {SIR} epidemic models.
\newblock {\em Math. Biosci.}, 292:86--96, 2017.
\newblock \url{https://doi.org/10.1016/j.mbs.2017.07.011}.

\bibitem{13}
R.~Bravo de~la Parra and L.~Sanz-Lorenzo.
\newblock Discrete epidemic models with two time scales.
\newblock {\em Adv. Difference Equ.}, 478:2021, 2021.
\newblock \url{https://doi.org/10.1186/s13662-021-03633-0}.

\bibitem{10}
C.~Castillo-Chavez, D.~Bichara, and B.R. Morin.
\newblock Perspectives on the role of mobility, behavior, and time scales in the spread of diseases.
\newblock {\em Proc. Natl. Acad. Sci. USA}, 113:14582--14588, 2016.
\newblock \url{https://doi.org/10.1073/pnas.1604994113}.

\bibitem{10418977}
J.~Chen, C.~Xia, and M.~Perc.
\newblock The {SIQRS} propagation model with quarantine on simplicial complexes.
\newblock {\em IEEE T. Comput. Soc. Syst.}, 11:4267--4278, 2024.
\newblock \url{https://doi.org/10.1109/TCSS.2024.3351173}.

\bibitem{32}
Y.~Chen and H.~Huang.
\newblock Modeling the impacts of contact tracing on an epidemic with asymptomatic infection.
\newblock {\em Appl. Math. Comput.}, 416:126754, 2022.
\newblock \url{https://doi.org/10.1016/j.amc.2021.126754}.

\bibitem{34}
P.~De~Maesschalck.
\newblock Smoothness of transition maps in singular perturbation problems with one fast variable.
\newblock {\em J. Differential Equations}, 244:1448--1466, 2008.
\newblock \url{https://doi.org/10.1016/j.jde.2007.10.023}.

\bibitem{35}
P.~De~Maesschalck and S.~Schecter.
\newblock The entry–exit function and geometric singular perturbation theory.
\newblock {\em J. Differential Equations}, 260:6697--6715, 2016.
\newblock \url{https://doi.org/10.1016/j.jde.2016.01.008}.

\bibitem{11}
R.~Della~Marca, A.~d’Onofrio, M.~Sensi, and S.~Sottile.
\newblock A geometric analysis of the impact of large but finite switching rates on vaccination evolutionary games.
\newblock {\em Nonlinear Anal. Real World Appl.}, 75:103986, 2024.
\newblock \url{https://doi.org/10.1016/j.nonrwa.2023.103986}.

\bibitem{23}
N.~Fenichel.
\newblock Geometric singular perturbation theory for ordinary differential equations.
\newblock {\em J. Differential Equations}, 31:53--98, 1979.
\newblock \url{https://doi.org/10.1016/0022-0396(79)90152-9}.

\bibitem{30}
G.~Gaeta.
\newblock A simple {SIR} model with a large set of asymptomatic infectives.
\newblock {\em Math. Eng.}, 3:1--39, 2021.
\newblock \url{https://doi.org/10.3934/mine.2021013}.

\bibitem{8}
A.~Goeke and C.~Lax.
\newblock Quasi-steady state reduction for compartmental systems.
\newblock {\em Phys. D}, 327:1--12, 2016.
\newblock \url{https://doi.org/10.1016/j.physd.2016.04.013}.

\bibitem{21}
D.~Greenhalgh.
\newblock Some results on optimal control applied to epidemics.
\newblock {\em Math. Biosci.}, 88:125--158, 1988.
\newblock \url{https://doi.org/10.1016/0025-5564(88)90040-5}.

\bibitem{24}
G.~Hek.
\newblock Geometric singular perturbation theory in biological practice.
\newblock {\em J. Math. Biol.}, 60:347--386, 2010.
\newblock \url{https://doi.org/10.1007/s00285-009-0266-7}.

\bibitem{7}
A.~Holmes, M.~Tildesley, and L.~Dyson.
\newblock Approximating steady state distributions for household structured epidemic models.
\newblock {\em J. Theoret. Biol.}, 534:110974, 2022.
\newblock \url{https://doi.org/10.1016/j.jtbi.2021.110974}.

\bibitem{39}
T.-H. Hsu and S.~Ruan.
\newblock Relaxation oscillations and the entry-exit function in multidimensional slow-fast systems.
\newblock {\em SIAM J. Math. Anal.}, 53:3717--3758, 2021.
\newblock \url{https://doi.org/10.1137/19M1295507}.

\bibitem{3}
H.~Jardón-Kojakhmetov, C.~Kuehn, A.~Pugliese, and M.~Sensi.
\newblock A geometric analysis of the {SIR}, {SIRS} and {SIRWS} epidemiological models.
\newblock {\em Nonlinear Anal. Real World Appl.}, 58:103220, 2021.
\newblock \url{https://doi.org/10.1016/j.nonrwa.2020.103220}.

\bibitem{14}
H.~Jardón-Kojakhmetov, C.~Kuehn, A.~Pugliese, and M.~Sensi.
\newblock A geometric analysis of the {SIRS} epidemiological model on a homogeneous network.
\newblock {\em J. Math. Biol.}, 83:37, 2021.
\newblock \url{https://doi.org/10.1007/s00285-021-01664-5}.

\bibitem{MR4362890}
M.~Jusup, P.~Holme, K.~Kanazawa, M.~Takayasu, I.~Romić, Z.~Wang, S.~Geček, T.~Lipić, B.~Podobnik, L.~Wang, W.~Luo, T.~Klanjšček, J.~Fan, S.~Boccaletti, and M.~Perc.
\newblock Social physics.
\newblock {\em Phys. Rep.}, 948:1--148, 2022.
\newblock \url{https://doi.org/10.1016/j.physrep.2021.10.005}.

\bibitem{36}
P.~Kaklamanos, C.~Kuehn, N.~Popović, and M.~Sensi.
\newblock Entry–exit functions in fast–slow systems with intersecting eigenvalues.
\newblock {\em J. Dyn. Diff. Equat.}, 37:559--–576, 2025.
\newblock \url{https://doi.org/10.1007/s10884-023-10266-2}.

\bibitem{15}
P.~Kaklamanos, A.~Pugliese, M.~Sensi, and S.~Sottile.
\newblock A geometric analysis of the {SIRS} model with secondary infections.
\newblock {\em SIAM J. Appl. Math.}, 84:661--686, 2024.
\newblock \url{https://doi.org/10.1137/23M1565632}.

\bibitem{5}
W.O. Kermack and A.G. McKendrick.
\newblock A contribution to the mathematical theory of epidemics.
\newblock {\em Proc. Roy. Soc. London Ser. A}, 115:700--721, 1927.
\newblock \url{https://doi.org/10.1098/rspa.1927.0118}.

\bibitem{40}
C.~Kuehn.
\newblock {\em Multiple Time Scale Dynamics}, volume 191.
\newblock Springer, Cham, 2015.
\newblock \url{https://doi.org/10.1007/978-3-319-12316-5}.

\bibitem{31}
K.Y. Leung, P.~Trapman, and T.~Britton.
\newblock Who is the infector? {E}pidemic models with symptomatic and asymptomatic cases.
\newblock {\em Math. Biosci.}, 301:190--198, 2018.
\newblock \url{https://doi.org/10.1016/j.mbs.2018.04.002}.

\bibitem{LI2025116273}
Y.~Li, Y.~Yao, M.~Feng, T.P. Benko, M.~Perc, and J.~Završnik.
\newblock Epidemic dynamics in homes and destinations under recurrent mobility patterns.
\newblock {\em Chaos, Solitons \& Fractals}, 195:116273, 2025.
\newblock \url{https://www.sciencedirect.com/science/article/pii/S0960077925002863}.

\bibitem{37}
W.~Liu.
\newblock Exchange lemmas for singular perturbation problems with certain turning points.
\newblock {\em J. Differential Equations}, 167:134--180, 2000.
\newblock \url{https://doi.org/10.1006/jdeq.2000.3778}.

\bibitem{20}
R.~Morton and K.H. Wickwire.
\newblock On the optimal control of a deterministic epidemic.
\newblock {\em Advances in Appl. Probability}, 6:622–--635, 1974.
\newblock \url{https://doi.org/10.2307/1426183}.

\bibitem{25}
V.~Nenchev.
\newblock Optimal quarantine control of an infectious outbreak.
\newblock {\em Chaos, Solitons \& Fractals}, 138:110139, 2020.
\newblock \url{https://doi.org/10.1016/j.chaos.2020.110139}.

\bibitem{29}
J.~Pan, Z.~Chen, Y.~He, T.~Liu, X.~Cheng, J.~Xiao, and H.~Feng.
\newblock Why controlling the asymptomatic infection is important: A modelling study with stability and sensitivity analysis.
\newblock {\em Fractal Fract.}, 6:197, 2022.
\newblock \url{https://doi.org/10.3390/fractalfract6040197}.

\bibitem{6}
A.T. Price-Smith.
\newblock {\em Contagion and Chaos: Disease, Ecology, and National Security in the Era of Globalization}.
\newblock The MIT Press, Cambridge, Massachusetts, 2009.
\newblock \url{https://doi.org/10.7551/mitpress/7390.001.0001}.

\bibitem{16}
P.~Rashkov and B.W. Kooi.
\newblock Complexity of host-vector dynamics in a two-strain dengue model.
\newblock {\em J. Biol. Dyn.}, 15:35--72, 2021.
\newblock \url{https://doi.org/10.1080/17513758.2020.1864038}.

\bibitem{17}
P.~Rashkov, E.~Venturino, M.~Aguiar, N.~Stollenwerk, and B.W. Kooi.
\newblock On the role of vector modeling in a minimalistic epidemic model.
\newblock {\em Math. Biosci. Eng.}, 16:4314--4338, 2019.
\newblock \url{https://doi.org/10.3934/mbe.2019215}.

\bibitem{38}
S.~Schecter.
\newblock Exchange lemmas. {II}. {G}eneral exchange lemma.
\newblock {\em J. Differential Equations}, 245(2):411--441, 2008.
\newblock \url{https://doi.org/10.1016/j.jde.2007.10.021}.

\bibitem{12}
S.~Schecter.
\newblock Geometric singular perturbation theory analysis of an epidemic model with spontaneous human behavioral change.
\newblock {\em J. Math. Biol.}, 82:54, 2021.
\newblock \url{https://doi.org/10.1007/s00285-021-01605-2}.

\bibitem{4}
Z.~Shuai and P.~van~den Driessche.
\newblock Global stability of infectious disease models using {L}yapunov functions.
\newblock {\em SIAM J. Appl. Math.}, 73:1513--1532, 2013.
\newblock \url{https://doi.org/10.1137/120876642}.

\bibitem{33}
A.K. Srivastav and M.~Ghosh.
\newblock Modeling and analysis of the symptomatic and asymptomatic infections of swine flu with optimal control.
\newblock {\em Model. Earth Syst. Environ.}, 2:1--9, 2016.
\newblock \url{https://doi.org/10.1007/s40808-016-0222-7}.

\bibitem{1}
P.~van~den Driessche and J.~Watmough.
\newblock Reproduction numbers and sub-threshold endemic equilibria for compartmental models of disease transmission.
\newblock {\em Math. Biosci.}, 180:29--48, 2002.
\newblock \url{https://doi.org/10.1016/S0025-5564(02)00108-6}.

\bibitem{0}
M.~Wechselberger.
\newblock {\em Geometric singular perturbation theory beyond the standard form}, volume~6.
\newblock Springer, 2020.
\newblock \url{https://link.springer.com/book/10.1007/978-3-030-36399-4}.

\bibitem{26}
X.~Yan and Y.~Zou.
\newblock Optimal and sub-optimal quarantine and isolation control in {SARS} epidemics.
\newblock {\em Math. Comput. Modelling}, 47:235--245, 2007.
\newblock \url{https://doi.org/10.1016/j.mcm.2007.04.003}.

\end{thebibliography}
\bigskip
\bigskip

\setlength\parindent{0pt}

\end{document}